\documentclass[10pt]{amsart}

\usepackage{amssymb}
\usepackage{amsmath}
\usepackage[initials]{amsrefs}
\usepackage{enumerate}
\usepackage{hyperref}

\usepackage[all]{xy}
\usepackage{pb-diagram,pb-xy}
\numberwithin{equation}{section}
 
\newtheorem{MainThm}{Theorem}
\newtheorem{Mainconjecture}[MainThm]{Conjecture}

\theoremstyle{definition}
\newtheorem{defn}[equation]{Definition}
\newtheorem{construction}[equation]{Construction}

\theoremstyle{remark}
\newtheorem{rem}[equation]{Remark}

\newtheorem{examples}[equation]{Examples}
\newtheorem{assumption}[equation]{Assumption}

\theoremstyle{plain}
\newtheorem{thm}[equation]{Theorem}
\newtheorem{lem}[equation]{Lemma}

\newtheorem{prop}[equation]{Proposition}
\newtheorem{cor}[equation]{Corollary}

\newcommand{\lra}{\longrightarrow}
\newcommand{\lla}{\longleftarrow}

\newcommand{\bC}{\mathbb{C}}
\newcommand{\bF}{\mathbb{F}}

\newcommand{\bP}{\mathbb{P}}
\newcommand{\bQ}{\mathbb{Q}}
\newcommand{\bR}{\mathbb{R}}

\newcommand{\bZ}{\mathbb{Z}}
\newcommand{\bH}{\mathbb{H}}
\newcommand{\bK}{\mathbb{K}}

\newcommand{\cA}{\mathcal{A}}

\newcommand{\Riem}{\mathcal{R}}

\newcommand{\Ad}{\mathrm{Ad}}

\newcommand{\fg}{\mathfrak{g}}

\newcommand{\ko}{\mathrm{ko}}
\newcommand{\KO}{\mathrm{KO}}

\newcommand{\im}{\operatorname{Im}}

\newcommand{\id}{\mathrm{Id}}

\newcommand{\ind}{\mathrm{ind}}
\newcommand{\inddiff}{\mathrm{inddiff}}

\newcommand{\res}{\mathrm{res}}

\newcommand{\cstar}{\mathrm{C}^{\ast}_{\mathrm{r}}}

\newcommand{\ahat}{\hat{\mathfrak{a}} }

\newcommand{\Stab}{\mathrm{Stab}}

\newcommand{\scpr}[1]{\langle #1 \rangle}

\newcommand{\rp}{\mathbb{R}\mathbb{P}}
\newcommand{\kp}{\mathbb{K}\mathbb{P}}

\newcommand{\torp}{\mathrm{tor}}
\newcommand{\dtorp}{\mathrm{dtor}}

\newcommand{\hp}{\mathbb{H}\mathbb{P}}
\newcommand{\cp}{\mathbb{C}\mathbb{P}}
\newcommand{\Spin}{\mathrm{Spin}}
\newcommand{\SO}{\mathrm{SO}}
\newcommand{\Sp}{\mathrm{Sp}}

\newcommand{\Fr}{\mathrm{Fr}}

\newcommand{\Aut}{\mathrm{Aut}}
\newcommand{\Isom}{\mathrm{Isom}}

\newcommand{\MSpin}{\mathrm{MSpin}}
\newcommand{\trf}{\mathrm{trf}}

\newcommand{\rank}{\mathrm{rank}}

\DeclareMathOperator{\scal}{scal}

\newcommand{\MU}{\mathrm{MU}}
\newcommand{\MSO}{\mathrm{MSO}}
\newcommand{\norm}[1]{\| #1 \|}
\newcommand{\hur}{\mathrm{hur}}

\title[Homotopy type of psc spaces]{On the homotopy type of the space of metrics of positive scalar curvature}

\author{Johannes Ebert, Michael Wiemeler}
\thanks{Both authors were supported by the Deutsche Forschungsgemeinschaft (DFG, German Research Foundation) -- Project-ID 427320536 -- SFB 1442, as well as under Germany’s Excellence Strategy EXC 2044 -- 390685587, Mathematics M\"unster: Dynamics–Geometry–Structure.
}

\email{johannes.ebert@uni-muenster.de}
\email{wiemelerm@uni-muenster.de}

\date{\today}

\begin{document}

\begin{abstract}
  Let \(M^d\) be a simply connected spin manifold of dimension \(d \geq 5\) admitting Riemannian metrics of positive scalar curvature.
  Denote by \(\Riem^+(M^d)\) the space of such metrics on \(M^d\).
  We show that \(\Riem^+(M^d)\) is homotopy equivalent to \(\Riem^+(S^d)\), where \(S^d\) denotes the \(d\)-dimensional sphere with standard smooth structure.

  We also show a similar result for simply connected non-spin manifolds \(M^d\) with \(d\geq 5\) and \(d\neq 8\).
  In this case let \(W^d\) be the total space of the non-trivial \(S^{d-2}\)-bundle with structure group \(SO(d-1)\) over \(S^2\). Then \(\Riem^+(M^d)\) is homotopy equivalent to \(\Riem^+(W^d)\).
\end{abstract}

\maketitle

\setcounter{tocdepth}{2}
\tableofcontents

\section{Introduction}

When studying Riemannian metrics of positive scalar curvature, there are two basic questions:
\begin{enumerate}
\item Given a closed manifold \(M^d\), is there a metric of positive scalar curvature on \(M^d\)?
\item If the answer to the first question is ``yes'', what can be said about the topology of the space \(\Riem^+(M^d)\) of positive scalar curvature on \(M^d\)?
\end{enumerate}

For simply connected manifolds of dimension \(d\geq 5\) the first question has the following answer: If \(M^d\) is non-spin then it always admits a metric of positive scalar curvature by results of Gromov and Lawson \cite{GL1} and Schoen and Yau \cite{SchoenYau}.

For spin manifolds $M^d$, there is an index-theoretic necessary condition for $M^d$ to admit a psc metric. Namely, the $\ahat$-invariant $\ahat (M) \in \ko_d$ has to be trivial ($\ahat(M)$ only depends on the spin cobordism class $[M]$ of $M$). A celebrated result of Stolz \cite{Stolz} states that this condition is also sufficient, when $M^d$ is simply connected and $d \geq 5$.

In this paper we study the second question.
Our first main result is as follows.

\begin{MainThm}\label{MainThm:spin}
Let $M$ be a simply connected closed spin manifold of dimension $d \geq 5$. Then if $M$ admits a psc metric, there is a homotopy equivalence
\[
\Riem^+ (M) \simeq \Riem^+ (S^d).
\]
Here \(S^d\) denotes the \(d\)-dimensional sphere with standard smooth structure.
\end{MainThm}

\begin{rem}
\begin{enumerate}
\item It has been known before that the homotopy type of $\Riem^+ (M)$ only depends on the spin cobordism class $[M] \in \Omega_d^\Spin$. More precisely, if $M,N$ are simply connected, spin and of dimension $d \geq 5$, and $[M]=[N]$, then $\Riem^+ (M) \simeq \Riem^+ (N)$. This is a consequence of Chernysh's refinement \cite{Chernysh} of the Gromov--Lawson surgery theorem \cite{GL1}, see \cite[Corollary 4.2]{WalshC}.  
\item It was shown before by Kordass \cite[Corollary 3.6]{Kordass} that $\Riem^+ (\hp^2) \simeq \Riem^+ (S^8)$, giving the first example of a homotopy equivalence between spaces of psc metrics on manifolds which are not cobordant. In fact, our proof of Theorem \ref{MainThm:spin} can be viewed as a generalization of the proof in \cite{Kordass}, and the main technical result of \cite{Kordass} (Theorem 3.1 of that paper) is a major ingredient for the present paper. 
\item Theorem~\ref{MainThm:spin} also holds in dimensions two and three because in these dimensions the spheres are the only simply connected manifolds by the Poincar\'e conjecture.
  In contrast to this it is completely open whether an analogue of the theorem holds in dimension four.
\end{enumerate}
\end{rem}

It is known that \(\Riem^+(S^d)\) has a very rich topology (though its homotopy type is not understood). There is a map
\[
\inddiff_{d,m}: \pi_m (\Riem^+(S^d)) \to \ko_{m+d+1}
\]
which is constructed using index theory, see \cite{Hitchin} or \cite{BERW}. For example, the following is known:

\begin{enumerate}
\item\label{item:0}  It is known that $\inddiff_{4k-1,0}$ is surjective if $k \geq 2$, see \cite{GL2} or \cite{kreckstolz2}. As $\ko_{4k}=\bZ$, this proves that \(\Riem^+(S^{4k-1})\) has infinitely many components when $k \geq 2$. This is in sharp contrast to the situation in low dimensions. Indeed, by \cite{Marques}, \(\Riem^+(S^3)\) is connected. Moreover, \(\Riem^+(S^2)\) is known to be contractible \cite[Theorem 3.4]{RosStolz}.
\item\label{item:1} The map $\inddiff_{d,m}$ is surjective after tensoring with $\bQ$ for all $d \geq 6$ and all $m\geq 1$ \cite{BERW}. For previous results in this direction see e.g. \cite{Hitchin}, \cite{HSS}, \cite{crowleyschick}, \cite{CSS}.
\item\label{item:3} \(\Riem^+(S^d)\), \(d\geq 3\), is homotopy equivalent to an \(H\)-space. In particular, \(\pi_1(\Riem^+(S^d,g_{S^d}))\) is abelian. Here \(g_{S^d}\) denotes the round metric on \(S^d\). Moreover, the component of \(g_{S^d}\) has the structure of an \(d\)-fold loop space \cite{WalshH-space}.
\item\label{item:2} The connected component of the round metric on \(S^d\), \(d\geq 6\), is homotopy equivalent to an infinite loop space \cite{ERWpsc3}. However, until now it is not known whether this loop space structure is compatible with the loop space structure in \eqref{item:3}.
\item The results mentioned in \eqref{item:0}, \eqref{item:1} were also known to be true for every other \(d\)-dimensional spin manifold admitting a psc metric. It follows from Theorem \eqref{MainThm:spin} that the results of \eqref{item:3} and \eqref{item:2} also hold for all simply connected spin manifolds admitting a psc metric.
\end{enumerate}

Let us turn to the non-spin case. If $M$ is simply connected, of dimension $d \geq 5$, and does not admit a spin structure, it admits a psc metric, by the above mentioned result of Gromov and Lawson. As an example for such a manifold, one can take the unit sphere bundle $W^d$ of the unique nontrivial vector bundle of rank $d-1$ over $S^2$. There is a similar cobordism invariance theorem for the homotopy type of $\Riem^+ (M)$ in that case: namely, if $M,N$ are both simply connected and oriented, do not admit spin structures, and their cobordism classes in the oriented cobordism group $\Omega_d^\SO$ agree, then $\Riem^+ (M) \simeq \Riem^+ (N)$ (see Theorem \ref{cobordismtheorem} below for a more general statement). We prove that more is true: 

\begin{MainThm}\label{MainThm:nonspin}
Let $M$ be a simply connected closed manifold of dimension $d \geq 5$ which does not admit a spin structure. Then if $d \neq 8$, there is a homotopy equivalence
\[
\Riem^+ (M) \simeq \Riem^+ (W^d).
\]
If $d=8$, there is either a homotopy equivalence $\Riem^+ (M) \simeq \Riem^+ (W^8)$ or $\Riem^+ (M) \simeq \Riem^+ (\cp^2 \times \cp^2)$. 
\end{MainThm}

\begin{rem}
\begin{enumerate}
\item There is no reason to assume that $\Riem^+ (W^d) \simeq \Riem^+ (S^d)$. On the other hand, there are currently no known techniques to disprove such a statement.
\item We conjecture that $\Riem^+ (\cp^2 \times \cp^2) \simeq \Riem^+ (W^8)$ as well; however our method misses this single case. 
\end{enumerate}
\end{rem}

Theorems \ref{MainThm:spin} and \ref{MainThm:nonspin} are special cases of the following more general

\begin{Mainconjecture}\label{conj:metaconjecture}
Let $M$ and $N$ be two closed $d$-manifolds, $d \geq 5$, with the same normal $2$-type. If both, $M$ and $N$, admit psc metrics, then $\Riem^+ (M) \simeq \Riem^+ (N)$.
\end{Mainconjecture}

Here two manifolds \(M_i\), \(i=1,2\), are said to have the same normal \(2\)-type, if there is a fibration \(\xi:B\rightarrow BO\) such that the classifying maps \(\nu_i:M_i\rightarrow BO\) of the stable normal bundles lift to maps \(\bar{\nu}_i: M_i\rightarrow B\) and both \(\bar{\nu}_i\) are \(2\)-connected.

\begin{rem}
\begin{enumerate}
\item It was observed in \cite[\S 9]{ERWpsc3} that Conjecture \ref{conj:metaconjecture} is implied by the concordance--implies--isotopy conjecture for psc metrics (which, however, is completely open). See Remark \ref{rem:concordance-isotopy} for an alternative explanation.
\item The method of the proof of Theorem \ref{MainThm:spin} also proves Conjecture \ref{conj:metaconjecture} for the normal $2$-type $B\Spin \times BG$ for some finitely presented groups, e.g. $G= \bZ^n$ or $G=F_n$ (free group on $n$ generators). That is, if $M$, $N$ are spin with fundamental group $G$, and both admit psc metrics, then $\Riem^+ (M) \simeq \Riem^+ (N)$. See Theorem \ref{thm:freeandfreeabelian} for a more general statement. 
\item Using the methods of this paper and computations in cobordism theory, one can prove partial results for many other normal $2$-types such as $BO$, $B\SO\times BG$ or $BO\times BG$. We refrain from stating them.
\end{enumerate}
\end{rem}

The proof of Theorem \ref{MainThm:spin} has three major steps. The first step, carried out in \S \ref{sec:cutandpaste}, is a general method to construct homotopy equivalences $\Riem^+ (M) \simeq \Riem^+ (N)$ when $M$ and $N$ are not cobordant. The key result is Theorem \ref{thm:admissiblesplittingleadstohomotopyinvariance}. It asserts that if a $d$-dimensional spin psc manifold $M$ admits a decomposition of a particular type (the technical name we chose for that is ``admissible splitting''), then for each simply connected $N$ in the cobordism class of $M$, we have $\Riem^+ (N) \simeq \Riem^+ (S^d)$. The second step is to show that total spaces of $\hp^2$-bundles with structure group $P(\Sp(2) \times \Sp(1))$ do admit admissible splittings. This follows from an extension of the Gromov--Lawson--Chernysh theorem that was proven by Korda\ss{} \cite{Kordass} and is done in \S \ref{sec:gromovlawsonchernysh}. 
The third step is to show that each simply connected spin manifold with psc metric is cobordant to the total space of an $\hp^2$-bundle with structure group $P(\Sp(2) \times \Sp(1))$. This is derived, by homotopy-theoretic methods, from the proof of Stolz' theorem \cite{Stolz} in \S \ref{sec:computationspin}. 

The proof of Theorem \ref{MainThm:nonspin} is somewhat more convoluted. The first step applies verbatim, and the geometric argument of \S \ref{sec:gromovlawsonchernysh} also apply to bundles whose fibers are complex projective spaces of (complex) dimension at least $3$. However, these examples do not suffice and we need a more explicit understanding of the structure of $\Omega_*^\SO$. Therefore, some extra arguments and a careful treatment of a particular $5$-manifold are necessary. Both is done in \S \ref{sec:geometry-for-SO}. We have not been able to do the extra work which would be needed for the $8$-dimensional case, which is why Theorem \ref{MainThm:nonspin} remains incomplete for $d=8$. 

\subsection*{Acknowledgements}
We would like to thank Jan-Bernhard Kordass on discussions on generalizations of the Gromov--Lawson--Chernysh surgery construction.

We also want to thank Michael Joachim, Christoph B\"ohm and Burkhard Wilking for helpful conversations about various arguments in this paper, and Georg Frenck for his careful reading of a first draft. 

\section{A cut-and-paste construction}\label{sec:cutandpaste}

\subsection{The cobordism theorem}

We have to review the role of cobordism theory for questions about positive scalar curvature and begin with the surgery theorem of Gromov--Lawson, Schoen--Yau and Chernysh. Let us recall that a \emph{torpedo metric $g_\torp^k$ of radius $\delta$} is an $O(k)$-invariant Riemannian metric on $\bR^k$ such that
\begin{enumerate}
\item in polar coordinates, $g_\torp^k$ is of the form $dr^2 \oplus f(r)^2 g_{S^{k-1}}$, where $0 \leq f \leq \delta$ is a smooth function with $f(0)=0$, $f'(0)=1$, $0 \leq f'\leq 1$, $f'' \leq 0$ and $f(t)=\delta$ for $t$ in a neighborhood of $[1,\infty)$. 
\item $\scal (g_\torp^k) \geq \scal(\delta^2 g_{S^{k-1}})= \frac{1}{\delta^2} (k-1)(k-2)$.
\end{enumerate}
Let $N$ be a closed manifold and let $\varphi:N \times \bR^k \to M$ be a codimension $0$ embedding. Fix a Riemannian metric $g_N$ on $N$, let $\delta$ be such that $\scal(g_N) + \frac{1}{\delta^2} (k-1)(k-2) >0 $ and fix a torpedo metric $g_\torp^k$ of radius $\delta>0$. We denote by $\Riem^+ (M, \varphi) \subset \Riem^+ (M)$ the subspace of those metrics which are of the form $g_N \oplus g_\torp^k$ on $N \times D^k$. 

\begin{thm}[Surgery theorem]\cite{Chernysh}\label{thm:surgerytheorem}
As long as $k \geq 3$, the inclusion 
\[
\Riem^+ (M,\varphi) \to \Riem^+ (M)
\]
is a weak homotopy equivalence.
\end{thm}

Gromov--Lawson \cite[Theorem A]{GL1} and Schoen--Yau \cite[Theorem 3]{SchoenYau} previously proved that $\Riem^+ (M,\varphi) \neq \emptyset$ if $\Riem^+ (M) \neq \emptyset$. A detailed account of Chernysh's proof appears also in \cite{EbFrenck}. 

To use this result effectively, we have to use structured cobordism. For a manifold $M$, let $\nu_M$ be its normal bundle (a stable vector bundle), and let $\gamma \to BO$ be the universal stable vector bundle. Let $\xi:  B \to BO$ be a fibration. By a $\xi$-manifold, we mean a manifold $M$, together with a map $\ell: M \to B$ and an isomorphism $(\xi \circ \ell)^* \gamma \cong \nu_M$ of stable vector bundles. The notion of $\xi$-cobordisms is defined similarly, and $\Omega_d^\xi$ is the cobordism group of $d$-dimensional $\xi$-manifolds. We say that a $\xi$-manifold $(M,\ell)$ is \emph{relatively $2$-connected} if $\ell$ is $2$-connected. A neccesary condition for the existence of a relatively $2$-connected $M$ is that $B$ is of type $(F_2)$ \cite{WallFin}, and if $d \geq 4$, this is also sufficient: under those assumptions, each class $x \in \Omega_d^\xi$ contains a relatively $2$-connected representative. 

The reader should have the following examples in mind; in each case, $B$ is of type $(F_2)$.
\begin{examples}
\begin{enumerate}
\item $\xi: B \Spin \to BO$. In that case, a $\xi$-manifold is the same as a spin manifold, and it is relatively $2$-connected if and only if it is simply connected. The cobordism groups are $\Omega_d^\Spin = \pi_d (\MSpin)$. 
\item $B \Spin \times BG \to B  \Spin \to BO$, where $G$ is a finitely presented group. A $\xi$-manifold is the same as a spin manifold equipped with a map $f: M \to BG$. The relevant cobordism group is $\Omega_d^\Spin (BG)= \pi_d (\MSpin \wedge BG_+)$, and a $\xi$-manifold is relatively $2$-connected if and only if it is $0$-connected and if $f$ induces an isomorphism on fundamental groups. 
\item $B = B \SO \to BO$. In that case, the relevant cobordism group is oriented cobordism $\Omega_d^{\SO}=\pi_d (\MSO)$. An oriented manifold is relatively $2$-connected if and only if it is simply connected and does not admit a spin structure. 
\item $B = BO$. The relevant cobordism group is unoriented cobordism $\Omega_d^{O}= \pi_d (\mathrm{MO})$. A manifold is relatively $2$-connected if it is $0$-connected, not orientable, has fundamental group $\bZ/2$ and the universal cover does not admit a spin structure. 
\end{enumerate}
\end{examples}

\begin{thm}[Cobordism theorem]\label{cobordismtheorem}
Let $(M_i,\ell_i)$ be two $\xi$-manifolds of dimension $d \geq 5$ and assume that $[M_0,\ell_0]=[M_1,\ell_1] \in \Omega_d^\xi$.
\begin{enumerate}
\item If $\Riem^+ (M_0)\neq \emptyset$ and $(M_1, \ell_1)$ is relatively $2$-connected, then $\Riem^+ (M_1) \neq \emptyset$. 
\item If both $(M_0,\ell_0)$ and $(M_1,\ell_1)$ are relatively $2$-connected, there is a weak homotopy equivalence
\[
 \Riem^+ (M_0) \simeq \Riem^+ (M_1). 
\]
\end{enumerate}
\end{thm}

This is a consequence of Theorem \ref{thm:surgerytheorem} and handle cancellation theory. The first part is due to Gromov and Lawson in the cases $B= B \Spin$ and $B= B\SO$ \cite{GL1}; the general version appears in \cite{HebestreitJoachim}, Theorem 2.1.1. The second part in the case $B= B \Spin$ was done in \cite[Corollary 4.2]{WalshC}; the general case requires slightly more elaborate arguments and is carried out in \cite[Theorem 1.5]{EbFrenck}. 

\subsection{Admissible splittings}

We now head towards a method which can be used to compare $\Riem^+ (M_0)$ and $\Riem^+ (M_1)$ when $M_0$ and $M_1$ are not cobordant. For the formulation, we shall make use of the notions of right stable and left stable psc metrics which were defined in \cite[Definition 1.2.1]{ERWpsc2}. To recall the definition, we have to introduce some notation and terminology. If $W$ is a compact manifold with (collared) boundary $M$, we denote by $\Riem^+ (W)$ the space of all psc metrics on $W$ which are of product form near $M$. There is a restriction map $\res:\Riem^+ (W)\to \Riem^+ (M)$ that takes the boundary metric, and for $h \in \Riem^+ (M)$, we let $\Riem^+ (W)_h:= \res^{-1}(h)$. 

If $W:M_0 \leadsto M_1$ is a cobordism and $h_i \in \Riem^+ (M_i)$, we let $\Riem^+ (W)_{h_0,h_1}:= \Riem^+ (W)_{h_0\coprod h_1}$. For each $g\in\Riem^+ (W)_{h_0,h_1}$, each cobordism $V: M_1 \leadsto M_2$ and each $h_2 \in \Riem^+ (M_2)$, there is a gluing map 
\[
\mu (g,\_): \Riem^+ (V)_{h_1,h_2} \to \Riem^+ (W \cup V)_{h_0,h_2}\quad \mu(g,h):=g\cup h,
\]
and we say that $g$ is \emph{right stable} if $\mu (g,\_)$ is a weak equivalence for all such $V$ and $h_2$. Left stability is formulated with the analogous gluing map $\mu (\_,g)$ defined for cobordisms with outgoing boundary $M_0$. A metric is \emph{stable} if it is both, left and right stable. We remark that $g$ is right stable if and only if it is left stable when considered as a metric on the reversed cobordism. As an example, consider $N \times D^k : \emptyset \leadsto N \times S^{k-1}$ for $k \geq 3$. If $g_N$ and $g_\torp^k$ are as in Theorem \ref{thm:surgerytheorem}, that result might be restated by saying that $g_N \oplus g_\torp^k\in \Riem^+(N\times D^k)_{g_N\oplus g_0}$ is right stable.

\begin{defn}\label{defn:splittingcondition-version1}
Let $M$ be a closed $d$-manifold. An \emph{admissible splitting} of $M$ is a decomposition of $M$ as the concatenation of two cobordisms
\[
\emptyset \stackrel{M_0}{\leadsto} N \stackrel{M_1}{\leadsto} \emptyset,
\]
such that the following conditions hold:
\begin{enumerate}
\item there are psc metrics $h_i \in \Riem^+ (N)$ and $g_i \in \Riem^+ (M_i)_{h_i}$, such that $h_0$ and $h_1$ lie in the same path component of $\Riem^+ (N)$,
\item $g_0$ is right stable and $g_1$ is left stable,
\item the inclusion maps $N \to M_i$ are $2$-connected. 
\end{enumerate}
\end{defn}
It is often more convenient to write such a splitting as $M= M_0 \cup_N M_1$. We remark that if $M$ has an admissible splitting, it automatically admits a psc metric (this is because we can find a psc metric in $\Riem^+ (N \times [0,1])_{h_0,h_1}$ if $h_0$ and $h_1$ lie in the same path component). 

\begin{examples}\label{examples:admissiblesplittings}
\begin{enumerate}
\item The splitting $S^d = D^d \cup_{S^{d-1}} D^d$ is admissible when $d \geq 3$: take $g_0=g_1 = g_\torp^{d}$. 
\item The standard splitting $S^d = (S^{k-1} \times D^{d-k+1}) \cup_{S^{k-1} \times S^{d-k}} (D^{k} \times S^{d-k})$ is admissible if $3 \leq k \leq d-2$: take $g_0= g_{S^{k-1}} \oplus g_\torp^{d-k+1}$ and $g_1 = g_\torp^k\oplus  g_{S^{d-k}}$. 
\item Let $P$ be a $d$-manifold, $d \geq 6$, such that $\partial P \to P$ is $2$-connected. Then by \cite[Theorem D]{ERWpsc2}, there is $g_{\partial P} \in \Riem^+ (\partial P)$ and a right stable $g_P \in \Riem^+ (P)_{g_{\partial P}}$. The decomposition $dP= P \cup_{\partial P} P^{op}$ of the double of $P$ is admissible. 
\end{enumerate}
\end{examples}

All these examples are nullbordant; and in fact the key point of this work is to construct admissible splittings of manifolds which are not nullbordant. Let us first see how admissible splittings are helpful for proving our main theorems. 

\begin{thm}\label{thm:admissiblesplittingleadstohomotopyinvariance}
Let $d \geq 5$ and let $\xi: B \to BO$ be a fibration. Let $x \in \Omega^\xi_d$ be a cobordism class which contains a representative $M$ that has an admissible splitting. Let $K$ and $L$ be relatively $2$-connected $\xi$-manifolds such that 
\[
 [L]= [K] + x \in \Omega_d^\xi.
\]
Then there is a weak equivalence
\[
\Riem^+ (L) \simeq \Riem^+ (K). 
\]
\end{thm}

It is obvious that the set of cobordism classes $x \in \Omega^\xi_d$ which contain a representative with an admissible splitting is an additive subgroup. For the proof of Theorem \ref{thm:admissiblesplittingleadstohomotopyinvariance}, we need a lemma. 

\begin{lem}\label{lem:fibrationtrick}
Let $M : \emptyset \leadsto N$ be a cobordism, let $g_0 \in \Riem^+ (M)$ be right stable and put $h_0:=g_0|_{N}$. Let $h_1 \in \Riem^+ (N)$ be any other psc metric in the path component of $h_0$. Then there is a right stable $g_1 \in \Riem^+ (M)$ with $g_1|_N=h_1$.  
\end{lem}

\begin{proof}
The restriction map $\res:\Riem^+ (M) \to \Riem^+ (N)$ is a Serre fibration by \cite[Theorem 1.1]{EbFrenck} (which is an improvement of the main result of \cite{Chernysh2}). Therefore, we can lift any path $t \mapsto h_t$ connecting $h_0$ and $h_1$ in $\Riem^+ (N)$ to a path $t \mapsto g_t$ in $\Riem^+ (M)$, starting at $g_0$. By \cite[Lemma 2.3.4]{ERWpsc3}, $g_1$ is right stable.
\end{proof}

\begin{proof}[Proof of Theorem \ref{thm:admissiblesplittingleadstohomotopyinvariance}]
Because of Lemma \ref{lem:fibrationtrick}, we may assume that $h_0=h_1=: h$. 

Without loss of generality, $B$ is of type $(F_2)$; otherwise, the Proposition holds vacuously. Let $M=M_0 \cup_N M_1$ be an admissible splitting of $M$, with psc metrics $g_0 \in \Riem^+ (M_0)_h$ and $g_1\in \Riem^+ (M_1)_h$ as in Definition \ref{defn:splittingcondition-version1}. Denote by \(M_0'\) the disjoint union of \(K\) with \(M_0\). 
Since $B$ is of type $(F_2)$, we can turn $M_0'$ by a sequence of $\xi$-surgeries of index $\leq 2$ in its interior to a relatively $2$-connected $\xi$-manifold $M''_0$ with boundary $N$. The $\xi$-manifold $M':=M''_0 \cup_N M_1$ lies in the cobordism class $x+[K]$, and is therefore $\xi$-cobordant to $L$. 

We show below that $M'$ is relatively $2$-connected. Taking this for granted, Theorem \ref{cobordismtheorem} shows that there is a weak equivalence
\begin{equation}\label{eqn:proofprop3}
\Riem^+ (L) \simeq \Riem^+ (M').
\end{equation}
Let further $M'':= M''_0 \cup_N M_0^{op}$. Since $g_1$ and $g_0^{op}$ are left stable, there are homotopy equivalences
\begin{equation}\label{eqn:proofprop1}
\Riem^+ (M') \stackrel{\mu(\_,g_1)}{\lla} \Riem^+(M''_0)_{h} \stackrel{\mu (\_, g_0^{op})}{\lra} \Riem^+ (M''),
\end{equation}
given by gluing in $g_1$ and $g_0^{op}$ respectively.
On the other hand
\[
[M'']= [K\#(M_0\cup_N M_0^{op})] = [K] \in \Omega_d^\xi;
\]
the second equality holds since the double of $M_0$ is nullbordant. 

We show below that $M''$ is relatively $2$-connected. Taking this for granted, Theorem \ref{cobordismtheorem} shows that there is a weak equivalence
\begin{equation}\label{eqn:proofprop2}
\Riem^+ (M'') \simeq \Riem^+ (K), 
\end{equation}
and combining \eqref{eqn:proofprop3}, \eqref{eqn:proofprop1} and \eqref{eqn:proofprop2} will finish the proof. 

It remains to be shown that $M'$ and $M''$ are relatively $2$-connected. Consider the commutative diagram
\[
\xymatrix{
M''_0 \ar[r]^{a} \ar[d]^{b} & M' \ar[d]^{\ell_{M'}} \\
M'' \ar[r]^{\ell_{M''}} & B. 
}
\]
The map $\ell_{M''_0}=\ell_{M'} \circ a = \ell_{M''} \circ b$ is $2$-connected. Furthermore, since $N \to M_i$ is $2$-connected, it follows by excision (or general position) that $a$ and $b$ are $2$-connected. Therefore $\ell_{M'}$ and $\ell_{M''}$ are $2$-connected as well.
\end{proof}

\begin{rem}\label{rem:concordance-isotopy}
Let $M$ be an arbitrary closed $d$-manifold with $d \geq 5$. Using Morse theory, we can write $M$ as composition $\emptyset \stackrel{M_0}{\leadsto} N \stackrel{M_1}{\leadsto} \emptyset$, where both inclusions $N \to M_i$ are $2$-connected. Using \cite[Theorem D]{ERWpsc2}, there are $h_0,h_1 \in \Riem^+ (N)$, a right stable $g_0 \in \Riem^+ (M_0)_{h_0}$ and a left stable $g_1 \in \Riem^+ (M_1)_{h_1}$. If $M$ admits a psc metric, the metrics $h_0,h_1$ are concordant. If we could infer from that information that $h_0$ and $h_1$ are isotopic (i.e. in the same path component), we would get an admissible splitting. 

In other words, an affirmative solution of the concordance--implies--isotopy conjecture (which is wide open of course) would imply immediately the truth of Conjecture \ref{conj:metaconjecture}. 

On the other hand, this observation means that when constructing admissible splittings, the key point is always to ensure that $h_0$ and $h_1$ are in the same path component.
\end{rem}

\section{Admissible splittings of projective bundles}\label{sec:gromovlawsonchernysh}

To apply Theorem \ref{thm:admissiblesplittingleadstohomotopyinvariance}, we need a sufficient supply of manifolds which have admissible splittings (note that all of the examples listed in \eqref{examples:admissiblesplittings} are nullbordant). The foremost goal of this section is the following result. 

\begin{thm}\label{thm:splitting-of-projective-bundles}
Let $E  \to B$ be a smooth fibre bundle with fibre $\kp^{m+n-1}$ over a closed manifold $B$ (here $\bK=\bR, \bC$ or $\bH$). Then $E$ has an admissible splitting in each of the following cases.
\begin{enumerate}
\item $\bK=\bR$, $m,n \geq 3$, and the structure group is $P (O(m) \times O(n)) = (O(m) \times O(n))/\{\pm 1\}$. 
\item $\bK=\bC$, $m,n \geq 2$, and the structure group is $U(m) \times U(n)$.
\item $\bK=\bC$, $m+n-1 \geq 3$, and the structure group is $C_2$, acting by complex conjugation.
\item $\bK=\bH$, $m,n \geq 1$, and the structure group is $P(\Sp(m) \times \Sp(n))= (\Sp(m) \times  \Sp(n))/\{\pm 1\}$. 
\end{enumerate}
\end{thm}

This goal is achieved in \S \ref{subsec:projectivebundles}, where we also introduce more uniform notation (the real case is not needed, but does not require any additional effort). The key to this result is a generalization of the surgery theorem to embeddings with nontrivial normal bundle, which was proven by Korda\ss{} in \cite{Kordass}. Precise statements are the key, and to that end, we need to review some basic material from Riemannian geometry. 

\subsection{Connection metrics}\label{subsec:connectionmetric}

\begin{assumption}\label{ass:data-for-connecttion}
Let $G$ be a Lie group with Lie algebra $\fg$ and let $p: P \to M$ be a smooth $G$-principal bundle, equipped with a connection $\theta $. Assume further that $F$ is a manifold with a left $G$-action. Finally, let $g_M$ be a Riemannian metric on $M$ and let $g_F $ be a $G$-invariant Riemannian metric on $F$. 
\end{assumption}

Out of these data, we produce a Riemannian metric $g_M \oplus_{\theta} g_F$ on the Borel construction $P \times_G F$. The construction is as follows. 

The datum of $\theta$ can be given by a certain $1$-form $\theta \in \cA^1 (P;\fg)$, or more geometrically by one of the following equivalent gadgets:
\begin{enumerate}
\item a $G$-equivariant left inverse $V$ of the inclusion map $\mathcal{V}_P := \ker (dp) \to TP$ of the vertical tangent bundle,
\item or by a $G$-invariant complementary subbundle $\mathcal{H}_P \subset TP$ to $\mathcal{V}_P$, the \emph{horizontal tangent bundle}.
\end{enumerate}
Given $V$, one puts $\mathcal{H}_P:= \ker (V)$ (this is also the same as $\ker (\theta)$). Given a decomposition $TP= \mathcal{V}_P \oplus \mathcal{H}_P$, one lets $V$ be the projection onto the first summand. There is also the horizontal projection $H: TP \to \mathcal{H}_P$; both $V$ and $H$ are $G$-equivariant. The differential of $p$ induces a $G$-equivariant isomorphism $\mathcal{H}_P \cong p^* TM$. 

Let $q: P \times F \to P \times_G F$ be the quotient map, and let $\pi: P \times_G F \to M$ be the projection to $M$. The subbundle $\mathcal{H}_P \times 0 \subset T(P \times F)$ projects to a subbundle 
\[
 q_* (\mathcal{H}_P \times 0)=: \mathcal{H}_{P \times_G F} \subset T(P\times_G F).
\]
Furthermore $d \pi$ restricts to a bundle isomorphism 
\[
 \mathcal{H}_{P \times_G F} \to \pi^* TM,
\]
and the kernel of $d \pi$ is the vertical tangent bundle $\mathcal{V}_{P \times_G F}= P \times_G TF$. Together, we obtain an isomorphism
\begin{equation}\label{eq:tangentbundle-borelconstruction}
\pi^* TM \oplus \mathcal{V}_{P \times_G F} \cong T(P \times_G F)
\end{equation}
which only depends on $\theta$.

\begin{defn}\label{defn:connectionmetric}
Let $G$, $P$, $M$, $F$, $\theta$, $g_M$ and $g_F$ as in \ref{ass:data-for-connecttion}. The \emph{connection metric} $g_M \oplus_\theta g_F$ is the Riemannian metric on $P \times_G F$ which under the isomorphism \eqref{eq:tangentbundle-borelconstruction} corresponds to the direct sum 
of $\pi^* g_M$ and the bundle metric on $\mathcal{V}_{P \times_G F} = P \times_G TF$ induced from the $G$-invariant metric $g_F$. 
\end{defn}

\begin{lem}\cite[Proposition 2.7.1]{GromollWalschap}
The map
\[
 \pi: (P \times_G F, g_M\oplus_\theta g_F) \to (M,g_M)
\]
is a Riemannian submersion with totally geodesic fibres.
\end{lem} 

The most common case arises from a rank $r$ Riemannian vector bundle $V \to M$ with a metric connection $\omega$ and an $O(r)$-invariant metric $g^r$ on $\bR^r$: the metric connection induces a connection on the $O(r)$-principal bundle $\Fr(V)$ of orthonormal frames of \(V\). 

Let us record an obvious, but useful naturality property of connection metrics.

\begin{lem}\label{lem:invariance-connectionmetric}
Let $G$, $P$, $M$, $F$, $\theta$, $g_M$ and $g_F$ as in \ref{ass:data-for-connecttion}. 
\begin{enumerate}
\item Let $H$ be a further Lie group which acts from the left on $P$ and $M$, compatible with the $G$-action on $P$, the map $p$, the Riemannian metric $g_M$ and the connection $\theta$. Then the connection metric $g_M \oplus_\theta g_F$ is $H$-invariant.
\item If $K$ is a Lie group that acts on $F$, commuting with the $G$-action and leaving $g_F $ invariant. Then $g_M \oplus_\theta g_F$ is $K$-invariant.
\end{enumerate}
\end{lem}

Assume the situation of (1), let $q:Q \to N$ be a smooth $H$-principal bundle with a connection $\omega$. In this situation, we can form the associated bundle 
\[
Q \times_H (P \times_G F) = (Q \times_H P) \times_G F \to N
\]
with fibre \(P\times_G F\).
Note that $Q \times_H P$ is a $G$-principal bundle over $Q \times_H M$, and it has an induced connection $Q \times_H \theta$, given by the following recipe. 
The subbundle $\mathcal{H}_Q \oplus \mathcal{H}_P \subset T(P \times Q)$ is $H \times G$-invariant and projects to an $G$-invariant subbundle $\mathcal{H}_{Q \times_H P} \subset T(Q \times_H P)$ which is complementary to $\mathcal{V}_{Q \times_H P}$.

\begin{lem}\label{lem:comparison-of-connection-metrics}
In the situation just described, the two metrics $(g_N \oplus_\omega  g_M) \oplus_{Q \times_H \theta} g_F$ and $g_N \oplus_\omega (g_M \oplus_\theta g_F)$ on $(Q \times_H P)\times_G F$ agree.
\end{lem}

\begin{proof}
Let $r: Q \times P \times F \to Z:= Q \times_H P \times_G F$ be the quotient map. It induces $G \times H$-equivariant isomorphisms
\[ 
q^* TN \times p^* TM \times TF \cong \mathcal{H}_Q \times \mathcal{H}_P \times TF  \cong r^* TZ. 
\]
Under this isomorphism, both connection metrics on $Z$ correspond to $g_N \oplus g_M \oplus g_F$.
\end{proof}

\subsection{The curvature of connection metrics}

We assume the notation introduced in \S \ref{subsec:connectionmetric}, and compute the scalar curvature of the connection metric $g_M \oplus_\theta g_F$ with the aid of O'Neill's formulas \cite{ONeill}, see also \cite{besse87}, Chapter 9. For a general Riemannian submersion $\pi:E \to M$, O'Neill defines two tensor fields $A$, $T \in \Gamma (E, T^* E^{\otimes 2} \otimes TE)$. Because the fibres of $\pi$ are totally geodesic in our situation, the tensor field $T$ vanishes \cite[\S 9.26]{besse87} and can be ignored. The $A$-tensor is defined by the formula (here $e,f$ are vector fields on $E$)
\[
 A_e f = V \nabla_{He} (Hf) + H \nabla_{He} (Vf). 
\]
\begin{lem}\label{lem:propertioes-A-tensor}
\begin{enumerate}
\item If $x,y$ are horizontal, $A_x y= \frac{1}{2} V [x,y]$. In particular, if \(E=P\times_G F\), $A_x y$ only depends on the connection $\theta$ and not on the involved metrics.
\item If $x$ is horizontal and $v$ is vertical, $A_x v$ is horizontal.
\item If $x,y$ are horizontal and $v$ is vertical, we have
\begin{equation}\label{eq:Atensorcalc}
\scpr{A_x v,y} + \scpr{v, A_x y} = 0.
\end{equation}
\end{enumerate}
\end{lem}

\begin{proof}
The first formula is \cite[Lemma 2]{ONeill}, and the second claim is clear. The third follows from
\[
 \scpr{A_x v,y} + \scpr{v, A_x y} = \scpr{H(\nabla_x v),y}+ \scpr{v, V (\nabla_x y)} =  \scpr{\nabla_x v,y} + \scpr{v, \nabla_x y} = x \scpr{v,y}=0.
\]
\end{proof}

For two horizontal vectors, we can express $A_x y$ in terms of the curvature of $\theta$, as follows. For $X \in \fg$, let $\rho(X)$ be the vector field on $F$ given by differentiating the $G$-action, and let $\Omega \in \cA^2 (P;\fg)$ be the curvature form of the connection $\theta$. For $q \in F$, let $\varphi: P \to P \times_G F$ be the map $p \mapsto [(p,q)]$. 

Let $x,y \in T_{[(p,q)]} (P \times_G F)$ be two horizontal vectors. By the definition of the horizontal distribution, we find (unique) horizontal tangent vectors $x_*,y_* \in T_p P$, such that $(T\varphi) x_* = x$ and $(T\varphi) y_* = y$. 

Using the formula \cite[p. 51]{Dupont} for the evaluation of $\Omega$ on horizontal vectors and \cite[Lemma 2]{ONeill}, we obtain the formula
\begin{equation}\label{eqn:computation-A-ternsor}
A_x y = - \rho(\Omega (x_*,y_*))_q.
\end{equation}

Now let $\sec_M$, $\sec_F$ and $\sec$ be the sectional curvatures of $g_M$, $g_F$ and $g_M \oplus_\theta g_F$, respectively. If $v,w$ are orthonormal vertical vectors and $x,y$ are orthonormal horizontal vectors, we have the following identities, by \cite[Corollary 1]{ONeill}: 
\begin{equation}\label{eq:ONeillformula}
\begin{split}
\sec (v,w) & = \sec_F (v,w), \\
\sec (x,v)  & = \norm{A_x v}^2 ,\\
\sec (x,y) & = \sec_M (x,y) - 3 \norm{A_x y}^2. 
\end{split}
\end{equation}
Now we can calculate the scalar curvature of $g_M \oplus_\theta g_F$. Let $(x_i)$ be an orthonormal basis for the horizontal vectors, and $(f_j)$ an orthonormal basis for the vertical vectors, both at a fixed point of $P \times_G F$. Then at this point

\begin{align*}
\scal (  g_M \oplus_\theta g_F)&= 2 \sum_{i<j} \sec (x_i,x_j) + 2 \sum_{i,j} \sec (x_i, f_j) +2 \sum_{i <j} \sec (f_i, f_j) \\
&=\scal (g_M) - 6 \sum_{i<j} \norm{A_{x_i} x_j}^2 + 2 \sum_{i,j} \norm{A_{x_i} f_j}^2 + \scal(g_F). 
\end{align*}

Since $A_{x_i} f_j$ is horizontal, its norm square is the same as 

\begin{align*}
  \sum_{i,j} \norm{A_{x_i} f_j}^2 &= \sum_{i,j,k} \scpr{A_{x_i} f_j, x_k}^2\\
  &=\sum_{i,j,k} \scpr{f_j,A_{x_i}  x_k}^2 && \text{(by (\ref{eq:Atensorcalc}))}\\
  &= \sum_{i,k} \norm{A_{x_i}  x_k}^2  && \text{(since $A_{x_i} x_j$ is vertical)}\\
  &=2 \sum_{i<k} \norm{A_{x_i}  x_k}^2 && \text{(by Lemma \ref{lem:propertioes-A-tensor} (1)).}
\end{align*}

Altogether, we obtain 
\begin{equation}\label{eq:scalarcurvature-of-connectionmetric}
\scal (  g_M \oplus_\theta g_F) = \scal (g_M) + \scal(g_F) -|A|^2, 
\end{equation}
where we denote 
\[
 |A|^2 := 2 \sum_{i<j} \norm{A_{x_i} x_j}^2 \geq 0.
\]
For $t>0$, we can scale the metric $g_F$ to $t g_F$. This does not change the orthonormal basis $(x_i)$, since $g_M$ is unchanged; so $A_{x_i} x_j$ is left unchanged, but its norm square gets scaled by $t$. Moreover, $\scal (t g_F)= \frac{1}{t} \scal (g_F)$. Altogether, we get
\begin{equation}\label{eq:scalarcurvature-scaling-identity}
\scal (  g_M \oplus_\theta t g_F) = \scal (g_M) + \frac{1}{t} \scal(g_F) -  t |A|^2 .
\end{equation}
More generally, with an additional scaling parameter, we have 
\begin{equation}\label{eq:scalarcurvature-scaling-identity2}
\scal ( s g_M \oplus_\theta t g_F) = \frac{1}{s} \scal (  g_M \oplus_\theta \frac{t}{s} g_F)
= \frac{1}{s} \scal (g_M) + \frac{1}{t} \scal(g_F) -  \frac{t}{s^2} |A|^2.
\end{equation}

\subsection{The generalized Gromov--Lawson--Chernysh theorem}

We can now state the generalized Gromov--Lawson--Chernysh surgery theorem \cite[Theorem 3.1]{Kordass} (see \cite[Theorem 3]{SchoenYau} for a precursor). Using the notion of right stability, the result might be stated as follows. 

\begin{thm}\cite[Theorem 3.1]{Kordass}\label{lemma:connectionmetricsrightstable}
Let $V \to N$ be a Riemannian vector bundle of rank $r \geq 3$ over a closed manifold $N$ with unit disc bundle \(D(V)\) and unit sphere bundle \(S(V)\), let $\omega$ be a metric connection on $V$ and let $h$ be a Riemannian metric on $N$.
If the connection metric $g = h \oplus_{\omega} g_{\torp}^r$ has positive scalar curvature, where $g_{\torp}^r$ is a torpedo metric, then $g$ is right stable on $D(V) \colon \emptyset \leadsto S(V)$.
\end{thm}

The following consequence of Theorem \ref{lemma:connectionmetricsrightstable} is our main tool to construct admissible splittings.

\begin{prop}\label{prop:janbernhard}
Assume that 
\begin{enumerate}
\item $V_i \to N_i$, $i=0,1$, are Riemannian vector bundles over closed manifolds, of rank $r_i := \rank(V_i) \geq 3$,
\item $\psi: S(V_0) \cong S(V_1)$ is a diffeomorphism of the sphere bundles, and $M:= D(V_0) \cup_{\psi} D(V_1)^{op}$, 
\item $G$ is a Lie group which acts smoothly on $N_i$, and on $V_i$ by bundle isometries, such that $\psi$ is $G$-equivariant (hence $G$ acts on $M$),
\item $\omega_i$, $i=0,1$, are $G$-equivariant metric connections on $V_i$ and $h_i$ are $G$-invariant Riemannian metric on $N_i$, and the ($G$-invariant) connection metrics $g_i := h_i \oplus_{\omega_i} g_{\torp}^{r_i}$ have positive scalar curvature (for suitable torpedo metrics $g_{\torp}^{r_i}$).
\item $\psi^* g_1|_{S(V_1)}$ and $g_0|_{S(V_0)}$ lie in the same path component of $\Riem^+ (S(V_0))^G$, the space of $G$-invariant psc metrics on $S(V_0)$. 
\item Let finally $P \to B$ be a smooth $G$-principal bundle over a closed manifold. 
\end{enumerate}
Then the manifold $P \times_G M$ has an admissible splitting.
\end{prop}

\begin{proof}
We can write
\begin{equation}\label{eq:proof-diskbundledec-admissible}
P \times_G M = (P \times_G D(V_0)) \cup_{P \times_G \psi} (P \times_G D(V_1)^{op}).
\end{equation}
The inclusions $S(V_i) \hookrightarrow D(V_i)$ are $2$-connected because $r_i \geq 3$, and hence so are the inclusions $P \times_G S(V_i) \to P \times_G D(V_i)$. 

We can write 
\[
P \times_G V_i = P \times_G (\Fr (V_i) \times_{O(r_i)} \bR^{r_i}) = (P \times_G \Fr(V_i)) \times_{O(r_i)} \bR^{r_i}. 
\]
Choose a Riemannian metric $g_B$ on $B$ and a connection $\theta$ on $P \to B$. We consider the metrics
\begin{equation}\label{eqn:proof-diskbundledec-admissible}
m_{i,u}:= (g_B \oplus_\theta u h_i) \oplus_{P \times_\theta \omega_i} u g_{\torp}^{r_i} \stackrel{\ref{lem:comparison-of-connection-metrics}}{=} 
g_B \oplus_\theta (u h_i \oplus_{\omega_i} u g_{\torp}^{r_i}) = g_B \oplus_\theta ug_i
\end{equation}
on $P \times_G V_i$ for a scaling parameter $u >0$. By the scaling identity \eqref{eq:scalarcurvature-scaling-identity}, there is $u_0$, such that for all $0< u \leq u_0$, $\scal(m_{i,u})>0$ (look at the right hand side of \eqref{eq:scalarcurvature-scaling-identity}). 

For such $u$, $m_{i,u}$ lies in the same path component of $\Riem^+ (P \times_G D(V_i))$ as 
\[
\frac{1}{u} m_{i,u} = (\frac{1}{u} g_B \oplus_\theta  h_i) \oplus_{P \times_\theta \omega_i}  g_{\torp}^{r_i} ,
\]
which is right stable by Theorem \ref{lemma:connectionmetricsrightstable} (where $P \times_G D(V)$ is viewed as a cobordism from $\emptyset$ to $P \times_G S(V)$). Using \cite[Lemma 2.3.4]{ERWpsc3}, it follows that $m_{i,u}$ is right stable. 

Let $k_i := g_i|_{S(V_i)}$ and let $\gamma: [0,1]  \to \Riem^+ (S(V_0))^G$ be a path connecting $k_0$ and $\psi^* k_1$. Without loss of generality, we can assume that $u_0$ is small enough so that $\scal (g_B \oplus_\theta u \gamma (t))>0$ for all $0< u \leq u_0$ and $t \in [0,1]$. 

For such a choice of $u$, we have a right stable metric $m_{0,u}$ on $Q_0:=P \times_G D(V_0)$, a left stable metric $m_{1,u}^{op}$ on $Q_1:=P \times_G D(V_1)^{op}$ and a path $\gamma': [0,1] \to \Riem^+ (P \times_G S(V_0))$ from $m_{0,u}|_{P \times_G S(V_0)}$ to $(P \times_G \psi)^* m_{1,u}$. This completes the proof.
\end{proof}

\subsection{Contruction of the splittings}\label{subsec:projectivebundles}

In this subsection, we finish the proof of Theorem \ref{thm:splitting-of-projective-bundles}. Let us first introduce uniform notation for the real, complex and quaternionic case.

\begin{enumerate}
\item $\bK$ is either $\bR$, $\bC$ or $\bH$, and $k := \dim_\bR(\bK)$. We denote by $\Im (\bK) \subset \bK$ the imaginary part, which we identify with the Lie algebra of the group $S^{k-1} \subset \bK^\times$. 
\item We consider $\bK^n$ as a \emph{right} $\bK$-vector space, equipped with the standard inner product. The group $S^{k-1}$ acts on $\bK^n$ and $S(\bK^n)= S^{kn-1}$ by right multiplication; let $R_z$ denote the right multiplication by $z \in S^{k-1}$. 
\item We denote by
\[
G(n):= N_{O(kn)} (S^{k-1})
\]
the normalizer subgroup of $S^{k-1}$ in $O(kn)$. The centralizer 
\[
H(n):= C_{O(kn)} (S^{k-1}) \subset G(n)
\]
is equal to $O(n), U(n), \Sp(n)$, depending on the case at hand. The action of $H(n)$ on $\bK^n$ by left matrix multiplication commutes with the $S^{k-1}$-action, and we let $L_A$ be the left multiplication by $A \in H(n)$. 
\item The quotient map $p:S^{kn-1} \to \kp^{n-1}:= S^{kn-1}/S^{k-1}$ to the projective space is an $S^{k-1}$-principal bundle. The projective space has an induced action of the normalizer $G(n)$, and the kernel of that $G(n)$-action is $S^{k-1}$, so that the group
\[
PG(n):= G(n) / S^{k-1}
\]
acts (faithfully) on $\kp^{n-1}$. 
\item More generally, we consider the subgroup $O(kn)\times O(km) \subset O(k(m+n))$, and define 
\[
G(n,m) := N_{O(kn) \times O(km)} (S^{k-1}) = G(n+m) \cap (O(kn) \times O(km)),
\]
\[
H(n,m):= C_{O(kn) \times O(km)} (S^{k-1}) = H(n) \times H(m) \subset G(n,m)
\]
and 
\[
PG(n,m):= G(n,m) / S^{k-1} \subset PG(n+m).
\]
\item The action of $G(n)$ on $S^{kn-1}$ is effective and transitive with stabilizer $\Stab_{G(n)}(e_n)= G(n-1)$. The action of $PG(n)$ on $\kp^{n-1}$ is transitive, and the stabilizer of $[e_n] \in \kp^{n-1}$ is equal to $PG(n-1,1)$. 
\item The structure of the group $G(n)$ is determined as follows. There is an exact sequence 
\[
1 \to H(n) \to G(n) \to \Aut (S^{k-1}) \to 1
\]
in each case. It follows that $G(n)=O(n)$ if $\bK=\bR$, that $G(n) = C_2 \ltimes U(n)$ in the complex case, where the cyclic group $C_2$ acts by complex conjugation, and that in the quaternionic case $G(n) = \Sp(n) \times_{\pm 1} \Sp(1)$, where $(A,z)$ act by $L_A R_z$. The quotient groups are $PG(n)=O(n)/\pm 1$ in the real case, $PG(n)=C_2 \ltimes U(n)/S^1$ in the complex case and $PG(n)= \Sp(n)/\pm 1=PSp(n)$ in the quaternionic case. 
\end{enumerate}

The goal of this section is the proof of the following result. 

\begin{prop}\label{prop:projectivebundles-decomposition}
Let $E \to B$ be a smooth fibre bundle with fibre $\kp^{m+n-1}$ and structure group $PG(n,m)$ over a closed manifold $B$. Then $E$ has an admissible splitting in each of the following cases.
\begin{enumerate}
\item $\bK=\bR$, $n,m \geq 3$, 
\item $\bK=\bC$, $n,m \geq 2$, 
\item $\bK=\bH$, $n,m \geq 1$.
\end{enumerate}
\end{prop}
Proposition \ref{prop:projectivebundles-decomposition} implies Theorem \ref{thm:splitting-of-projective-bundles}, because each of the structure groups mentioned in that theorem acts through its image in $PG(n,m)$. 

We start the proof of Proposition \ref{prop:projectivebundles-decomposition} with a suitable decomposition of $\kp^{n+m-1}$ into disc bundles which is invariant under $PG(n,m)$. 
Let 
\[
L_0 := \{ (v,w) \in S(\bK^n \times \bK^m) \vert \norm{v} \geq \norm{w}\},  \;  L_1 := \{ (v,w) \in S(\bK^n \times \bK^m)  \vert \norm{v} \leq \norm{w}\}
\]
with intersection 
\[
L_{01}:= L_0 \cap L_1 = \{ (v,w) \vert \norm{v}^2 = \norm{w}^2 = \frac{1}{2}\} \cong S(\bK^n) \times S(\bK^m). 
\]
The decomposition $S^ {kn+km-1} = L_0 \cup L_1$ is invariant under the action of $O(kn) \times O(km)$ and of $S^{k-1}$, and hence in particular under $G(n,m)$. It therefore induces a decomposition $\kp^{n+m-1}= P_0 \cup_{P_{01}} P_1$, where $P_i := L_i /S^{k-1}$, which is invariant under $PG(n,m)$. 
The map 
\[
\varphi_0: S(\bK^n) \times_{S^{k-1}} D(\bK^m) \to \kp^{n+m-1}; \; [v,w] \mapsto [\frac{1}{\sqrt{2}} (v,w)]
\]
is a diffeomorphism onto $P_0$, and there is a similar diffeomorphism 
\[
\varphi_1:  S(\bK^m) \times_{S^{k-1}} D(\bK^n) \cong P_1. 
\]
The diffeomorphisms $\varphi_i$ are $PG(n,m)$-equivariant. Altogether, we have found a $PG(n,m)$-invariant decomposition 
\begin{equation}\label{eq:decomposition-kpn-topology}
\kp^{n+m-1} = D (S(\bK^n) \times_{S^{k-1}}\bK^m) ) \cup_{\varphi_1^{-1}  \circ \varphi_0} D (S(\bK^m) \times_{S^{k-1}}\bK^n) )
\end{equation}
into disc bundles. If $n$ and $m$ are as in Proposition \ref{prop:projectivebundles-decomposition}, these disc bundles have rank at least $3$. The action of $PG(n,m)$ on the sphere bundle $S (S(\bK^n) \times_{S^{k-1}}\bK^m)$ is transitive. This sphere bundle is more conveniently written as 
\[
M_{k,n,m}:= S^{kn-1} \times_{S^{k-1}} S^{km-1}  .
\]
Note that $M_{k,n,m}$ is both, an $S^{km-1}$-bundle over $\kp^{n-1}$ and an $S^{kn-1}$-bundle over $\kp^{m-1}$. The action of the group $PG(n,m)$ on $M_{k,n,m}$ is transitive, because the action of $H(n,m)$ on $S^{kn-1} \times S^{km-1}$ is already transitive. 

Let $g_{S^{kn-1}}$ be the round metric on $S^{kn-1}$. The action of $O(kn)$ and hence of $G(n)$ is by isometries. The action of the subgroup $S^{k-1} \subset G(n)$ is geodesic in the sense that orbits of $1$-parameter subgroups of $S^{k-1}$ are geodesics in $S^{kn-1}$. We let $\mathcal{V}_{S^{kn-1}}$ be the $(k-1)$-dimensional distribution of tangent vectors to orbits. The distribution $\mathcal{V}_{S^{kn-1}}$ is preserved by $G(n)$, and so is its orthogonal complement $\mathcal{H}_{S^{kn-1}}:= (\mathcal{V}_{S^{kn-1}})^\bot$. Moreover $\mathcal{H}_{S^{kn-1}}$ is the horizontal distribution for a principal connection $\omega_n$ on the $S^{k-1}$-principal bundle $p: S^{kn-1} \to \kp^{n-1}$. In fact, $\omega_n$ is $G(n)$-equivariant, and it is the only connection with this property. 

The \emph{Fubini-Study metric} on $\kp^{n-1}$ is the quotient metric $g_{\kp^{n-1}}$ of the round metric. The map $\pi:(S^{kn-1},g_{S^{kn-1}}) \to (\kp^{n-1},g_{\kp^{n-1}})$ is a Riemannian submersion with totally geodesic fibres, and $g_{S^{kn-1}}$ can be identified with the connection metric 
\[
 g_{S^{kn-1}} = g_{\kp^{n-1}} \oplus_{\omega_n } g_{S^{k-1}}
\]
(since both, vertical and horizontal distributions are orthogonal for both metrics, and both metrics agree on horizontal and vertical vectors). 

The metric $g_{\kp^{n-1}} \oplus_{\omega_n} tg_\torp^{km}$ on $S^{kn-1} \times_{S^{k-1} } D^{km}= P_0$ is $PG(n,m)$-invariant, and so is the metric $g_{\kp^{m-1}} \oplus_{\omega_m} sg_\torp^{kn}$ on $S^{km-1} \times_{S^{k-1} } D^{kn}= P_1$. For $t,s$ sufficiently small, these have positive scalar curvature (provided $n$ and $m$ are as in Proposition \ref{prop:projectivebundles-decomposition}). By Proposition \ref{prop:janbernhard}, the proof of Proposition \ref{prop:projectivebundles-decomposition} is complete, once we can show that for suitable $s$ and $t$, the two metrics 
\[
 g_{\kp^{n-1}} \oplus_{\omega_n} tg_{S^{km-1}} ; \quad  g_{\kp^{m-1}} \oplus_{\omega_m} sg_{S^{kn-1}} \in \Riem^+ (M_{k,n,m})^{PG(n,m)}
\]
can be connected by a path of $PG(n,m)$-invariant psc metrics. Therefore, all that remains to be shown is:

\begin{lem}\label{lem:contractibility-invariantmetrics}
The space of $PG(n,m)$-invariant psc metrics on $M_{k,n,m}$ is contractible, if $k=1$ and $m,n \geq 3$, or $k=2$ and  $m,n \geq 2$, or $k=4$ and $m,n \geq 1$.
\end{lem}

The proof of Lemma \ref{lem:contractibility-invariantmetrics} will be quite direct, but a bit lengthy, and we begin with some preliminaries. 
\begin{defn}
The \emph{Berger metric} $g_{s,t}^{k,n}$ on $S^{kn-1}$ with parameters $s,t >0$ is 
\[
 g_{s,t}^{k,n} := s g_{\kp^{n-1}} \oplus_{\omega_n} t g_{S^{k-1}}. 
\]
\end{defn}
The Berger metrics are $G(n)$-invariant, and are the only such metrics (we shall not need this fact explicitly). 
The metric 
\[
s g_{\kp^{n-1}} \oplus_{\omega_n} g_{u,t}^{k,m}
\]
on $M_{k,n,m}$ is $PG(n,m)$-invariant, and similarly is 
\[
s g_{\kp^{m-1}} \oplus_{\omega_m} g_{u,t}^{k,n}.
\]
In fact, it is not hard to see that 
\[
s g_{\kp^{n-1}} \oplus_{\omega_n} g_{u,t}^{k,m} = u g_{\kp^{m-1}} \oplus_{\omega_m} g_{s,t}^{k,n},
\]
but we do not make use of this fact. 

\begin{lem}\label{lem:classification-of-invariant-metrics}
All $PG(n,m)$-invariant metrics on $M_{k,n,m}$ are of the form $s g_{\kp^{n-1}} \oplus_{\omega_n} g_{u,t}^{k,m}$.
\end{lem}

\begin{proof}
The map 
\[
q: M_{k,n,m} \to \kp^{n-1} \times \kp^{m-1}
\]
is an $S^{k-1}$-principal bundle and $PG(n,m)$-equivariant. The tangent bundle of $M_{k,n,m}$ can be decomposed into three summands:
\[
TM_{k,n,m} = \mathcal{V}_{ M_{k,n,m}} \oplus q^{*} (T\kp^{n-1}\oplus 0) \oplus q^{*} (0 \oplus T \kp^{m-1}),
\]
and this decomposition is $PG(n,m)$-equivariant. We shall now argue that any $PG(n,m)$-invariant symmetric bilinear form on $TM_{k,n,m}$ is the direct sum of invariant bilinear forms on the three summands, and that each summand carries a unique (up to scalar multiple) such invariant symmetric bilinear form. From this, the claim follows by counting parameters (for that, note that the construction of $s g_{\kp^{n-1}} \oplus_{\omega_n} g_{u,t}^{k,m}$ can be carried out for arbitrary real $s,u,t$ and gives an invariant bilinear form). 

For that goal, pick the basepoint $[e_n,e_m] \in M_{k,n,m}$ and let $K(n,m) \subset PG(n,m)$ be the stabilizer subgroup at $[e_n,e_m]$. Because the action of $PG(n,m)$ on $M_{k,n,m}$ is transitive, $PG(n,m)$-invariant bilinear forms on $M_{k,n,m}$ correspond bijectively to $K(n,m)$-invariant bilinear forms on the tangent space $T_{[e_n,e_m]} M_{k,n,m}$. 
The stabilizer $K(n,m)$ is equal to
\begin{align*}
\{ (A,z,B,z) \in O(n-1)\times O(1) \times O(m-1)\times O(1)  \}/\{\pm 1\} & & \bK=\bR,\\
C_2 \ltimes \{ (A,z,B,z) \in U(n-1)\times U(1) \times U(m-1)\times U(1)  \}/U(1) & & \bK=\bC,\\
\{ (A,z,B,z) \in \Sp(n-1)\times \Sp(1) \times \Sp(m-1)\times \Sp(1)  \} / \{\pm 1\} & &\bK=\bH.
\end{align*}

The tangent space $T_{(e_n,e_m)} (S^{kn-1} \times S^{km-1})$ can be identified with 
\[
\bK^{n-1} \oplus \Im (\bK) \oplus \bK^{m-1} \oplus \Im (\bK),
\]
and similarly, we have 
\begin{equation}\label{eq:tangent-space-spherebundle-decomposition}
T_{[e_n,e_m]} (S^{kn-1} \times_{S^{k-1}} S^{km-1})=\bK^{n-1} \oplus \bK^{m-1} \oplus \Im (\bK). 
\end{equation}
This identification is $K(n,m)$-equivariant, where the action on the right hand side is as follows. 
\begin{enumerate}
\item If $\bK=\bR$, we have 
\[
T_{[e_n,e_m]} (S^{n-1} \times_{S^{0}} S^{m-1})=\bR^{n-1} \oplus \bR^{m-1} ,
\]
and the action of $[A,z,B,z] \in K(n,m)$ is by $L_A R_z$ in the first summand, and by $L_B R_z$ in the second one. The two summands are irreducible inequivalent representations of $O(n-1) \times 1 \times O(m-1) \times 1 \subset K(n,m)$. 
\item If $\bK= \bC$, we have 
\[
T_{[e_n,e_m]} (S^{2n-1} \times_{S^1} S^{2m-1})=\bC^{n-1} \oplus \bC^{m-1}  \oplus \Im (\bC). 
\]
The action of $[A,z,B,z] \in K(n,m)$ is by $L_A R_{\overline{z}}$ in the first summand, by $L_B R_{\overline{z}}$ in the second one, and by the identity on the third. The action of the complex conjugation $\kappa \in K(n,m)$ is by complex conjugation in all three summands. Again, the decomposition \eqref{eq:tangent-space-spherebundle-decomposition} is a decomposition into irreducible, pairwise inequivalent $K(n,m)$-representations. 
\item If $\bK=\bH$, we have 
\[
T_{[e_n,e_m]} (S^{4n-1} \times_{S^3} S^{4m-1})=\bH^{n-1} \oplus \bH^{m-1} \oplus \Im (\bH). 
\]
The action of $[A,z,B,z] \in K(n,m)$ is by $L_A R_{\overline{z}}$ in the first summand, by $L_B R_{\overline{z}}$ in the second summand, and by $\Ad (z)$ in the third one. Again, this is a decomposition into pairwise inequivalent representations. 
\end{enumerate}
Since the summands in the decomposition \eqref{eq:tangent-space-spherebundle-decomposition} are irreducible and inequivalent $K(n,m)$-representations, it follows that all invariant bilinear forms are given by direct sums of invariant forms on the three summands, and the forms on the summands are unique up to scalar multiplication. Hence the (real) dimension $d_{k,n,m}$ of the space of $PG(n,m)$-invariant bilinear forms on $M_{k,n,m}$ is equal to 
\[
d_{k,n,m}= 
\begin{cases}
0 & \bK=\bR, \, n=m=1\\
1 & \bK=\bR, \, \text{either} \, n = 1 \,\text{ or }\, m=1\\
2 & \bK=\bR, \, n,m \geq 2\\
1 & \bK=\bC, \bH,\, n=m=1\\
2 & \bK=\bC,\bH, \, \text{either} \, n = 1 \,\text{ or }\, m=1\\
3 & \bK=\bC, \bH,\, n,m \geq 2.\\
\end{cases}
\]
A parameter count shows that this is the same as the number of parameters in $s g_{\kp^{n-1}} \oplus_{\omega_n} g_{u,t}^{k,m}$.
\end{proof}

What remains to be proven is that 
\[
\{ (s,u,t) \in (0,\infty)^3 \vert \scal (s g_{\kp^{n-1}} \oplus_{\omega_n} g_{u,t}^{k,m}) >0\} \subset \bR^3
\]
is contractible. To write down a formula for $\scal (s g_{\kp^{n-1}} \oplus_{\omega_n} g_{u,t}^{k,m})$, we fix some more notation. We define
\begin{equation}\label{eqn:scalarcurvature-spheres}
c_n := \scal (g_{S^{n-1}}) = (n-1)(n-2)
\end{equation}
and 
\begin{equation}\label{eqn:scalar-curvature-projectivespace}
b_{k,n} := \scal (g_{\kp^{n-1}}) = 
 \begin{cases}
  c_n & \bK=\bR,\\
  4n(n-1) & \bK=\bC,\\
  16 (n^2-1) & \bK=\bH  
 \end{cases}
\end{equation}
(for the last one, see \cite{Gray}). Furthermore, we let $a_{k,n}$ be the (constant) value of the quantity $|A|^2$ (see \eqref{eq:scalarcurvature-of-connectionmetric}) for the Riemannian submersion $(S^{kn-1},g_{S^{kn-1}}) \to (\kp^{n-1}, g_{\kp^{n-1}})$. Using \eqref{eq:scalarcurvature-of-connectionmetric}, we find
\[
 c_{kn}=\scal( g_{S^{kn-1}})= \scal (g_{\kp^{n-1}} \oplus_{\omega_n} g_{S^{k-1}}) = b_{k,n} + c_k - a_{k,n}, 
\]
from which we may evaluate
\begin{equation}\label{eqn:A-torsor-hopf-bundles}
a_{k,n}=
\begin{cases}
0 & \bK=\bR,\\
2 (n-1) & \bK=\bC,\\
12 (n-1) & \bK=\bH.
\end{cases}
\end{equation}

\begin{lem}\label{lem:curvature-connection-metric-berger}
With these notations, we have
\[
 \scal ( g_{s,t}^{k,n} ) = \frac{1}{s} b_{k,n} + \frac{1}{t} c_k - \frac{t}{s^2} a_{k,n}
\]
and
\[
\scal (sg_{\kp^{n-1}} \oplus_{\omega_n} g_{u,t}^{k,m}) =    \frac{1}{s}   b_{k,n}+ \frac{1}{u} b_{k,m} + \frac{1}{t} c_k - \frac{t}{u^2} a_{k,m} - \frac{t}{s^2} a_{k,n} .
\]
\end{lem}

\begin{proof}
The formula for the scalar curvature of the Berger metric is a straightforward consequence of \eqref{eq:scalarcurvature-scaling-identity2}. To get the second formula, let $a_{k,n,m,u,t}$ be the (constant) value of $|A|^2$ for the connection metric $g_{\kp^{n-1}} \oplus_{\omega_n} g_{u,t}^{k,m}$. This can be computed with the aid of the formula \eqref{eqn:computation-A-ternsor}. Because the action vector field of the $S^{k-1}$-action on $S^{km-1}$ lies in the vertical tangent bundle $\mathcal{V} S^{km-1}$ which is scaled by the factor $t$, the quantity $a_{k,n,m,u,t}$ does not depend on $u$; the dependence on $t$ is expressed by the formula
\[
 a_{k,n,m,u,t} = t a_{k,n,m,1,1}. 
\]
Moreover, if $S^{km-1}$ carries the round metric, the norm of the action vector field by an element in the Lie algebra $\Im (\bK)$ of $S^{k-1}$ agrees with the norm of the action vector field on $S^{k-1}$ (for $p \in S^{km-1}$ and $X \in \Im (\bK)$, the curve $t \mapsto p\exp(tX)$ is a closed geodesic with length independent of $m$). Therefore, $a_{k,n,m,1,1}$ does not depend on $m$ and is given by $a_{k,n,m,1,1} = a_{k,n}$. This shows
\[
a_{k,n,m,u,t} = t a_{k,n },
\]
and we conclude
\[
 \scal (g_{\kp^{n-1}} \oplus_{\omega_n} g_{u,t}^{k,m}) =  b_{k,n}+ \frac{1}{u} b_{k,m} + \frac{1}{t} c_k - \frac{t}{u^2} a_{k,m} - t a_{k,n}.
\]
By global scaling with $\frac{1}{s}$, we obtain the formula for $\scal (sg_{\kp^{n-1}} \oplus_{\omega_n} g_{u,t}^{k,m})$, using \eqref{eq:scalarcurvature-scaling-identity2}.
\end{proof}

\begin{proof}[Proof of Lemma \ref{lem:contractibility-invariantmetrics}]
Each $PG(n,m)$-invariant metric can be written as $sg_{\kp^{n-1}} \oplus_{\omega_n} g_{u,t}^{k,m}$, by Lemma \ref{lem:classification-of-invariant-metrics}. By global scaling and Lemma \ref{lem:curvature-connection-metric-berger}, the space of invariant psc metrics is homotopy equivalent to the space of all $(s,u) \in (0,\infty)^2$ such that 
\begin{equation}\label{eq.scalar-curvature-invariant}
\frac{1}{s}   b_{k,n}+ \frac{1}{u} b_{k,m} +  c_k - \frac{1}{u^2} a_{k,m} - \frac{1}{s^2} a_{k,n}>0 .
\end{equation}
In the cases excluded in the statement of the Lemma, the left hand side of \eqref{eq.scalar-curvature-invariant} is always $0$. If $\bK=\bR$, we have $c_1 = a_{1,n}= 0 $, and hence the left hand side of \eqref{eq.scalar-curvature-invariant} is equal to $\frac{1}{s}   b_{k,n}+ \frac{1}{u} b_{k,m} $, which is always positive. 

For the remaining cases, change coordinates to $x := \frac{1}{s}$ and $y: = \frac{1}{u}$, and let 
\[
 q(x,y) :=  a_{k,n} x^2+  a_{k,m} y^2,
\]
\[
 c := c_k \geq 0
\]
and 
\[
 \ell (x,y) :=   b_{k,n} x+ b_{k,m} y. 
\]
We have to look at the set of all $(x,y) \in (0,\infty)^2$ such that 
\begin{equation}\label{eq:inequality}
q(x,y) - \ell(x,y) < c.
\end{equation}
Now $q$ is a positive semidefinite quadratic form and $\ell$ is a linear form. Hence $q-\ell$ is a convex function, and so the set of solutions to the inequality \eqref{eq:inequality} is convex. It is also nonempty: if $\bK=\bH$, $c>0$ and so \eqref{eq:inequality} is satisfied for sufficiently small $(x,y)$. If $\bK=\bC$, $c=0$, we find solutions of \eqref{eq:inequality}, provided that one of $b_{k,n} $, $ b_{k,m}$ is positive, which is the case unless $n=m=1$. 
\end{proof}

This concludes the proof of Proposition \ref{prop:projectivebundles-decomposition}. 

\section{The computation in the spin case}\label{sec:computationspin}

\subsection{The simply connected case}

In this section, we finish the proof of Theorem \ref{MainThm:spin}. 

\begin{prop}\label{prop:reducingstructuregroup}
The map $\pi: B P(\Sp(2) \times \Sp(1)) \to B P \Sp(3)$ induces a surjective map
\[
\pi_*: \Omega_*^{\Spin}(B P(\Sp(2) \times \Sp(1))) \to \Omega_*^\Spin( B P \Sp(3)).
\]
\end{prop}

Let us postpone the proof for a moment. 

\begin{proof}[Proof of Theorem \ref{MainThm:spin}, assuming Proposition \ref{prop:reducingstructuregroup}]
We first have to recall some aspects of the proof of Stolz' theorem \cite{Stolz} (that a simply connected spin manifold of dimension $d \geq 5$ admits a psc metric if $\ahat(M)=0$). The Atiyah--Bott--Shapiro orientation of spin vector bundles gives rise to a map
\[
 \ahat: \MSpin \to \ko
\]
of ring spectra, from the spin cobordism spectrum to the connective real $K$-theory spectrum. On the level of homotopy groups, $\ahat$ induces
\[
 \Omega_d^\Spin \to \ko_d
\]
which assigns to a spin manifold the index of the spin Dirac operator, by the Atiyah--Singer index theorem. The Lichnerowicz formula implies that $\ahat(M)=0$ is a necessary condition for the existence of a psc metric on $M$. Now let 
\[
 T: \Omega_{d-8}^\Spin (B P \Sp(3)) \to \Omega_d^\Spin 
\]
be the map that assigns to a map $f:B \to P \Sp(3)$ the total space of the $\hp^2$-bundle classified by $f$. The composition $\ahat \circ T$ is zero, since total spaces of $\hp^2$-bundles admit psc metrics. The key result of \cite{Stolz} is that the sequence
\begin{equation}\label{eqn:stolzkrecksequence}
 \Omega^\Spin_{d-8}(B P \Sp(3)) \stackrel{T}{\to} \Omega^\Spin_d \stackrel{\ahat}{\to} \ko_d 
\end{equation}
is exact (it is \emph{not} part of a long exact sequence). To give credits properly, Stolz \cite[Theorem 1.2]{Stolz} proved exactness of \eqref{eqn:stolzkrecksequence} after localization at the prime $2$. After inverting $2$, this was shown later by Kreck and Stolz \cite[Proposition 3.3]{KreckStolz}. 

Now let $M^d$ be a spin manifold which admits a psc metric, $d \geq 5$. Then $\ahat(M)=0$, and so by the exactness of \eqref{eqn:stolzkrecksequence}, $[M] \in \Omega^\Spin_d$ contains the total space of an $\hp^2$-bundle with structure group $P \Sp (3)$. By Proposition \ref{prop:reducingstructuregroup}, $[M]$ contains the total space of an $\hp^2$-bundle with structure group $P (\Sp(2) \times \Sp(1))$. 
By Theorems \ref{thm:splitting-of-projective-bundles} and \ref{thm:admissiblesplittingleadstohomotopyinvariance}, it follows that $\Riem^+ (M) \simeq \Riem^+ (S^d)$, as claimed, if in addition \(M\) is simply connected. 
\end{proof}

\begin{proof}[Proof of Proposition \ref{prop:reducingstructuregroup}] 
We shall prove separately that $\pi_* \otimes \bZ_{(2)}$ and $\pi_* \otimes \bZ[\frac{1}{2}]$ are surjective. We first turn to the $2$-local case. Because 
\begin{equation}\label{eqn:hptwohomogeneous}
\hp^2 = P\Sp(3)/ P (\Sp(2)\times \Sp(1)),
\end{equation}
 $\pi$ is homotopy equivalent to  a smooth fibre bundle with fibre $\hp^2$. In this situation, Becker and Gottlieb \cite{BeckerGottlieb} constructed a transfer map
\[
 \trf_\pi: \Sigma^\infty  B P \Sp(3)_+ \to \Sigma^\infty B P(\Sp(2) \times \Sp(1))_+
\]
on the level of suspension spectra, and they showed that the composition $\pi \circ \trf_\pi: \Sigma^\infty B P \Sp(3)_+ \to  \Sigma^\infty B P \Sp(3)_+$ induces multiplication by the Euler number $\chi(\hp^2)=3 \in \bZ_{(2)}^\times$ on (ordinary) homology. If $A_*$ is a generalized homology theory whose coefficient groups are $\bZ_{(2)}$-modules, a straightforward argument with the Atiyah--Hirzebruch spectral sequence proves from this fact that $\pi_*:A_* (B P(\Sp(2) \times \Sp(1))) \to A_* ( B P \Sp(3))$ is (split) surjective. Apply that statement to $A_* := \Omega_*^\Spin (\_) \otimes \bZ_{(2)}$. There was nothing special about spin cobordism in this argument. 

For the consideration after inverting $2$, some more ingredients are needed. Firstly, the map $B \Sp(3) \to B P \Sp(3)$ induces an isomorphism in homology with coefficients in $\bZ[\frac{1}{2}]$, and hence also in spin bordism tensored with $\bZ[\frac{1}{2}]$. The same argument applies to $B (\Sp(2) \times \Sp(1)) \to B P (\Sp(2) \times \Sp(1))$. Therefore, we only need to prove that 
\begin{equation}\label{eqn:surjetive-invert2}
\Omega^\Spin_* (B (\Sp(2) \times \Sp(1))) \otimes \bZ[\frac{1}{2}] \to \Omega^\Spin_* (B \Sp(3)) \otimes \bZ[\frac{1}{2}]
\end{equation}
is surjective. It is well-known that the integral homology of $B \Sp (n)$ is torsionfree and concentrated in degrees divisible by $4$, see e.g. \cite[Theorem III.5.6]{MimuraToda}. Moreover, the map $B (\Sp(2) \times \Sp(1)) \to B \Sp(3)$ induces an epimorphism in homology with arbitrary coefficients. This is because the homotopy fibre is $\hp^2$; hence the Leray--Serre spectral sequence collapses (for degree reasons) at the $E^2$-page.

We also need to know a fact about spin cobordism from Anderson, Brown and Peterson's work \cite{AnBrPet}. Namely, Theorem 2.2 of that paper implies that $\Omega_*^\Spin \otimes \bZ[\frac{1}{2}]$ is concentrated in degrees divisible by $4$ (this can also be deduced from the structure of $\Omega^{\SO}_*$ described in Theorem \ref{thm:structuremso1} below, together with the fact that the forgetful map $\Omega^{\Spin}_* \otimes \bZ[\frac{1}{2}]  \to \Omega^{\SO}_* \otimes \bZ[\frac{1}{2}] $ is an isomorphism, which follows from Serre class theory). Therefore, the Atiyah--Hirzebruch spectral sequence
\[
 E^2_{p,q} = H_p (B \Sp(3);\Omega_q^\Spin \otimes \bZ[\frac{1}{2}]) \Rightarrow \Omega^\Spin_{p+q}(B\Sp(3)) \otimes \bZ[\frac{1}{2}] 
\]
and its analogue $\bar{E}^2_{p,q}$ for $B( \Sp(2) \times \Sp(1))$ collapse at the $E^2$-page. We have argued that the comparison map $\bar{E}^2_{p,q} \to E^2_{p,q}$ is surjective, and since the spectral sequences collapse, surjectivity on the $E^\infty$-page follows. Finally, an argument with the $5$-lemma proves that \eqref{eqn:surjetive-invert2} is surjective, as claimed.
\end{proof}

\subsection{Generalization to some nontrivial fundamental groups}

If one tries to carry over the above argument for Theorem \ref{MainThm:spin} to the case of spin manifolds with a given fundamental group $\Gamma$, one runs into formidable difficulties. These difficulties are very similar to that which are met when studying the Gromov--Lawson--Rosenberg conjecture, which has counterexamples for rather innocent groups \cite{Schick}. But at least in some interesting cases, our argument goes through. We have not attempted a more systematic study, though. 

\begin{thm}\label{thm:freeandfreeabelian}
Let $\Gamma$ be a finitely presented torsion-free group, and assume the following:
\begin{enumerate}
\item the classifying space $B \Gamma$ splits stably as a wedge of sphere spectra, in other words, $\Sigma^\infty B \Gamma_+ \simeq \bigvee_{i \in I}\Sigma^{n_i} \mathbb{S}$,
\item $\Gamma$ satisfies the real strong Novikov conjecture, in other words, the assembly map $\mu:\KO_* (B\Gamma) \to \KO_* (\cstar(\Gamma))$ is injective. 
\end{enumerate}
Then if $M_0,M_1$ are two closed spin manifolds of dimension $d \geq 5$ with fundamental group $\Gamma$ and $\Riem^+ (M_i)\neq \emptyset$, we have
\[
\Riem^+ (M_0) \simeq \Riem^+ (M_1). 
\]
\end{thm}

\begin{examples}
\begin{enumerate}
\item The free abelian group $\bZ^n$ satisfies all requirements. The stable splitting condition is clear, and free abelian groups even satisfy the Baum--Connes conjecture (in the complex case), which implies the real strong Novikov conjecture by \cite{SchickKKK}. 
\item Similarly, the free group $F_n$ satisfies the hypotheses of the Theorem (free groups satisfy the Baum--Connes conjecture by \cite{HigsonKasparov} because they have the Haagerup property). 
\end{enumerate}
\end{examples}

\begin{proof}
We have to prove that the cobordism class $[M_i] \in \Omega_d^\Spin (B \Gamma)$ contains a representative which is an $\hp^2$-bundle with structure group $H:= P(\Sp(2) \times \Sp(1))$.

The hypothesis that $\Sigma^\infty B \Gamma_+ \simeq \bigvee_{i \in I}\Sigma^{n_i} \mathbb{S}$ has the effect that 
\[
A_k (B \Gamma) = \bigoplus_{i \in I} A_{k-n_i} (*)
\]
for each homology theory $A_*$, and this sum decomposition is natural in $A_*$. Kreck--Stolz' sequence \eqref{eqn:stolzkrecksequence}, together with Proposition \ref{prop:reducingstructuregroup} and the (obvious) injectivity of $\ko_* \to \KO_*$ yields an exact sequence 
\[
\Omega_{*-8}^\Spin (B P(\Sp(2) \times \Sp(1))) \to \Omega_*^\Spin \to \KO_*
\]
whose maps are induced by map of spectra. The induced sequence 
\[
\Omega_{*-8}^\Spin (B P(\Sp(2) \times \Sp(1)) \times B \Gamma) \to \Omega_*^\Spin (B \Gamma)\to \KO_* (B \Gamma)
\]
is therefore exact, by the above observation. Assuming that $\Gamma$ satisfies the real strong Novikov conjecture gives an exact sequence 
\begin{equation}\label{eqn:pione-sequence}
\Omega_{*-8}^\Spin (B P(\Sp(2) \times \Sp(1)) \times B \Gamma) \stackrel{T}{\to} \Omega_*^\Spin (B \Gamma)\stackrel{\ind_\Gamma}{\to} \KO_* (C^*_r (\Gamma)).
\end{equation}
The map $T$ again takes the total space of the classified $\hp^2$-bundle, and $\ind_\Gamma$ associates to a cobordism class $(M,f) $ the index of the Rosenberg--Dirac operator. The Lichnerowicz vanishing argument carries over to this case and shows that if $M$ admits a psc metric, then $\ind_\Gamma (M,f)=0$. Using the exactness of \eqref{eqn:pione-sequence}, we conclude that $(M,f)$ is cobordant to the total space of an $\hp^2$-bundle with structure group $P(\Sp(2) \times \Sp(1))$. 
\end{proof}

\section{Two further geometric ingredients for the non-spin case}\label{sec:geometry-for-SO}

For the proof of Theorem \ref{MainThm:spin} that we just gave, Theorem \ref{thm:splitting-of-projective-bundles} provided enough examples for admissible splittings. For the proof of Theorem \ref{MainThm:nonspin}, these do not quite suffice, and we need some more geometric ingredients, which will be developed in this section. The main results are Propositions \ref{cor:decomposables-have-admissibledecomposition}, \ref{prop:diskbundle-cross-pscmanifolds} and \ref{prop:lastcase}. 

\subsection{Stability of product metrics}\label{sec:products}

The goal of this subsection is the proof of the following two propositions.

\begin{prop}\label{cor:decomposables-have-admissibledecomposition}
Let $x \in \Omega^{\SO}_*$ be a cobordism class which can be written as $x= yz$, with $y$ of dimension at least $6$ and $z$ of positive dimension. Then $x$ contains a representative with an admissible splitting. 
\end{prop}

\begin{prop}\label{prop:diskbundle-cross-pscmanifolds}
Let $V_i \to N_i$ be vector bundles of rank $r_i \geq 3$, let $\psi: S(V_0) \cong S(V_1)$ be a diffeomorphism and let $M := D(V_0) \cup_{\psi} D(V_1)^{op}$. Let $P$ be a closed manifold which admits a psc metric. Then $M \times P$ has an admissible splitting. 
\end{prop}

An important step in the proof of Proposition \ref{cor:decomposables-have-admissibledecomposition} is the following result.

\begin{lem}\label{prop:stability-of-productmetrics}
Let $W:M_0 \leadsto M_1$ be a cobordism such that $\dim (W) =d \geq 6$ and such that both inclusions $M_i \to W$ are $2$-connected. Let $(N,g_N)$ be a closed psc manifold, and let $g_W$ be a Riemannian metric on $W$ which has product form near the boundary, so that the product metric $g_W \oplus g_N$ on $W \times N$ has positive scalar curvature. Then $g_W \oplus g_N$ is stable.
\end{lem}

The proof of Lemma \ref{prop:stability-of-productmetrics} begins with two lemmas. 

\begin{lem}\label{lem:almostpscms-contractible}
Let $W$ be a compact manifold with boundary. Then for each $\epsilon>0$, the space 
\[
\Riem^{> - \epsilon} (W) := \{ h \in \Riem (W) \vert \inf \scal (h) > - \epsilon \}
\]
is (weakly) contractible. If $(N,g_N)$ is a closed psc manifold, the space of product psc metrics (that is, psc metrics of the form $g_W \oplus g_N$) on $W \times N$ is contractible.
\end{lem}

\begin{proof}
The space $\Riem^{> - \epsilon} (W)$ is nonempty: if $h \in \Riem (W)$, then $th \in \Riem^{> - \epsilon} (W)$ if $t$ is sufficiently large. Pick $h_0 \in \Riem^{> - \epsilon} (W)$. Let $X$ be a compact space and let $h: X \to  \Riem^{> - \epsilon} (W)$ be a map. Let 
\[
H : K \times I \to \Riem (W), \; H(x,t):= (1-t)h(x) + th_0, 
\]
and 
\[
\lambda:=  \inf_{(x,t ) \in K \times [0,1]} \scal (H(x,t)) . 
\]
Choose $a \geq 1$ such that $\frac{1}{a} \lambda > - \epsilon$, and define a homotopy $P:K \times [0,3] \to \Riem^{> - \epsilon} (W)$ by
\[
P(x,t) := 
\begin{cases}
(1-t+ ta ) h(x)  & t \in [0,1]\\
(2-t) ah(x) + (t-1) ah_0 & t \in [1,2]\\
((3-t) a+ (t-2)) h_0  & t \in [2,3].
\end{cases}
\]
Hence $\Riem^{>-\epsilon}(W) \simeq \ast$. The second claim of the Lemma follows from the first and the observation that 
\[
\scal (g_W \oplus g_N) >0 \Leftrightarrow \inf \scal (g_W) > - \inf \scal (g_N). 
\]
\end{proof}

\begin{lem}\label{lem:stability-of-priductmetrics-preservedundersurgeires}
Let $N$ be a closed manifold which admits a psc metric, let $W^d : M_0 \leadsto M_1$ be a cobordism, and assume that each product psc metric on $W \times N$ is right stable. Let $\varphi: S^{k-1} \times D^{d-k+1} \to W$ be a surgery datum in the interior of $W$ with $3 \leq k \leq d-2$, and let $W_\varphi$ be the result of performing a surgery along $\varphi$. Then each product psc metric on $W_\varphi \times N$ is right stable. The same is true for left stable metrics. 
\end{lem}

\begin{proof}
Let $\varphi': S^{d-k} \times D^k \to W_\varphi$ be the dual surgery datum. We define 
\[
\phi:= \varphi \times \id_N:  S^{k-1} \times D^{d-k+1}  \times N \to W \times N
\]
and 
\[
\phi' := \varphi'\times \id_N: S^{d-k} \times D^k \to W_\varphi \times N. 
\]
Then 
\[
W_\varphi \times N = ((W \times N) \setminus \im (\phi)) \cup_{S^{k-1} \times S^{d-k} \times N} (D^k \times S^{d-k} \times N). 
\]
Fix a psc metric $g_N$ on $N$ and let $\Riem^+ (W \times N, \phi)  \subset \Riem^+ (W \times N)$ be the subspace of those metrics which are equal to $g_{S^{k-1}} \oplus g_\torp^{d-k} \oplus g_N$ on $\im (\phi)$, and define $\Riem^+ (W_\varphi \times N, \phi')  \subset \Riem^+ (W_\varphi \times N)$ in an analogous way. The inclusions 
\[
\Riem^+ (W \times N, \phi)  \to \Riem^+ (W \times N), \; \; \Riem^+ (W_\varphi \times N, \phi')  \to \Riem^+ (W_\varphi \times N)
\]
are weak equivalences by Theorem \ref{thm:surgerytheorem} (or rather a variant of that for manifolds with boundary which is covered by the proof in \cite{EbFrenck}). There is a homeomorphism
\begin{equation}\label{eq:generalized-surgery}
\Riem^+ (W \times N, \phi) \cong \Riem^+ (W_\varphi \times N, \phi'),
\end{equation}
and hence a weak equivalence
\[
S: \Riem^+ (W \times N) \simeq \Riem^+ (W_\varphi \times N).
\]
By the argument given in the proof of Theorem 2.6 of \cite{BERW}, the path components which contain right stable metric correspond to each other under the bijection $\pi_0 (S)$. Furthermore, the homeomorphism \eqref{eq:generalized-surgery} and its inverse preserve product metrics. By Lemma \ref{lem:almostpscms-contractible}, there are \emph{unique} components of $\Riem^+ (W \times N) $ and $ \Riem^+ (W_\varphi \times N)$ which contain metrics of the form $g_W \oplus g_N$, and so these two components correspond to each other under $\pi_0 (S)$. Hence one is right stable if and only if the other one is.
\end{proof}

\begin{proof}[Proof of Lemma \ref{prop:stability-of-productmetrics}]
This is similar to the proof of \cite[Theorem E = Theorem 3.1.2(ii)]{ERWpsc2}, and we refer to that paper for more details. First note that any psc metric of the form $g_M \oplus dt^2 \oplus g_N$ on $M \times [0,1] \times N$ is stable. Now let $W$ be as in the statement of the proposition. By the proof of the $s$-cobordism theorem \cite{Kervaire} or \cite{WallGC}, $W$ has a handlebody decomposition (relative to $M_0$) in which all handles have index in $\{3 ,\ldots, d-2\}$. As explained in the proof of \cite[Theorem 3.1.2(ii)]{ERWpsc2}, it follows that $W \cup W^{op}$ and $M_0 \times [0,2]$ are related by a sequence of surgeries of index in $\{ 3, \ldots, d-2\}$. By Lemma \ref{lem:stability-of-priductmetrics-preservedundersurgeires}, each product psc metric on $(W \cup W^{op}) \times N$ is stable. Now let $h \in \Riem^+ (W \times N)$ be a product psc metric. Then $h \cup h^{op}$ is a product psc metric, therefore stable. Thus the gluing map $\mu ( \_,h)$ has a left inverse in the homotopy category, namely $\mu (\_,h^{op})$. 

The same argument can be applied with the roles of $W$ and $W^{op}$ reversed, and leads to the conclusion that $\mu (\_,h^{op})$ also has a left inverse in the homotopy category. Therefore $\mu (\_,h^{op})$ is a weak equivalence, and so is $\mu (\_,h)$. This means that $h$ is left stable. By Theorem E of \cite{ERWpsc2}, $h$ is also right stable. 
\end{proof}

\begin{proof}[Proof of Proposition \ref{cor:decomposables-have-admissibledecomposition}]
Let $M^d$ be a representative for $y$ which is relatively $3$-connected. In particular, $M$ is simply connected, $H_2 (M);\bZ)\cong H_2 (B\SO; \bZ) \cong \bZ/2$, and $M$ does not have a spin structure. 

Let $S^2 \to M$ be an embedding which represents a generator of $\pi_2 (M)\cong \bZ/2$. The normal bundle is the nontrivial vector bundle over $S^2$, and so we obtain a $2$-connected embedding $P \to M$ of the disc bundle of the nontrivial vector bundle over $S^2$, which leads to a decomposition $M = P \cup_{\partial P} (M\setminus P)$. By general position, the inclusion map $M\setminus P \to M$ is $(d-3) \geq 3$-connected, and so the manifold $M \setminus P$ is relatively $3$-connected.

The argument so far only used that $M$ was relatively $3$-connected and not that it was closed. Hence we can apply the same argument to $M\setminus P$. The result is a decomposition 
\[
M = P \cup W \cup P^{op}
\]
where $W: \partial P \leadsto \partial P$ is a cobordism such that both inclusions $\partial P \to W$ are $2$-connected. Let $g_P$ be a connection torpedo metric on $P$, and choose $g_W \in \Riem(W)_{g_{\partial P}, g_{\partial P}}$, not necessarily of positive scalar curvature. 

The cobordism class $z \in \Omega_*^{\SO}$ contains a representative $N$ which carries a psc metric $g_N$. This follows from the proof of \cite[Corollary C]{GL1} (a more elegant way to see this is to invoke the main result of \cite{Fuehring}). We can assume that $\scal (g_W \oplus g_N)>0$.

By Lemma \ref{prop:stability-of-productmetrics}, the metric $g_W \oplus g_N$ is stable, and the metric $g_P \oplus g_N$ is a connection torpedo metric, so it is right stable by Theorem \ref{lemma:connectionmetricsrightstable}. Likewise, $g_P^{op} \oplus g_N$ is left stable, and we have found an admissible splitting
\[
M \times N = (P \times N) \cup_{\partial P \times N} ((W \cup P^{op}) \times N)
\]
with suitable psc metric $(g_P \oplus g_N) \cup ((g_W \cup g_P^{op}) \oplus g_N$. 
\end{proof}

For the proof of Proposition \ref{prop:diskbundle-cross-pscmanifolds}, we need another lemma. 

\begin{lem}\label{lem:psc-cross-nullhomotopic}
Let $(M,g_M)$ and $(N,g_N)$ be two psc manifolds. Then the product map $\nu:\Riem^+ (M) \times \Riem^+ (N) \to \Riem^+ (M \times N)$, $(g_0, g_1) \mapsto g_0 \oplus g_1$ is homotopic to the constant map. 
\end{lem}

\begin{proof}
We only show that the map is weakly homotopic, i.e. that the restriction to a compact subset $K  \subset \Riem^+ (M) \times \Riem^+ (N)$ is nullhomotopic. The technique to write down the homotopies is very similar to the proof of Lemma \ref{lem:almostpscms-contractible}, hence we shall be very brief. 
The first step is a homotopy from $\nu|_K$ to the map $(g_0,g_1) \mapsto g_M \oplus g_1$, which is very similar to the homotopy used in the proof of Lemma \ref{lem:almostpscms-contractible}. The second step is a homotopy from that map to the constant map $g_M \oplus g_N$. 
\end{proof}

\begin{proof}[Proof of Proposition \ref{prop:diskbundle-cross-pscmanifolds}]
Pick Riemannian metrics $h_i$ on $N_i$ such that $h_i \oplus g_P$ is a psc metric on $N_i \times P$. Pick metric connections $\omega_i$ on $V_i$ and suitable torpedo metrics $g_\torp^{r_i}$ on $\bR^{r_i}$ such that $h_i \oplus_{\omega_i} g_\torp^ {r_i}$ is a psc metric on $D(V_i)$. 

The product metrics $(h_i \oplus_{\omega_i} g_\torp^ {r_i}) \oplus g_P$ are connection torpedo metrics. Lemma \ref{lem:psc-cross-nullhomotopic} implies that the metrics $(h_0 \oplus_{\omega_0} g_\torp^ {r_0})|_{S(V_0)} \oplus g_P$ and $(\psi^* (h_1 \oplus_{\omega_1} g_\torp^ {r_1})|_{S(V_1)}) \oplus g_P$ are isotopic; this finishes the proof because of Proposition \ref{prop:janbernhard}. 
\end{proof}

\subsection{Splitting a certain \texorpdfstring{$5$}{5}-manifold}\label{sec:wu}

Here, we prove the following. 

\begin{prop}\label{prop:lastcase}
There is a closed oriented $5$-manifold $M$ with the following properties:
\begin{enumerate}
\item We can write $M= W_0 \cup W_1$ as a union of two disc bundles of rank $3$ over $S^2$, 
\item this is an admissible splitting, and 
\item $[M] \neq 0 \in \Omega_5^\SO$.
\end{enumerate}
\end{prop}

\begin{construction}\label{construction-wu-manifold}
Let 
\[
V := S^3 \times_{S^1} (\bC \oplus \bR) \to \cp^1,
\]
 where $S^1$ acts on $\bC$ by rotations and trivially on $\bR$. This is the unique nontrivial real vector bundle of rank $3$ on $\cp^1$, and $V$ does not have a spin structure. We write $W:= D(V)$ and $N:= S(V)$ for the unit disc bundle and unit sphere bundle. Note that
\[
N= S^3 \times_{S^1} S^2
\]
is the \emph{Hirzebruch surface}; in \cite{Hirzebruch1951}, it is shown that $N \cong \cp^2 \# \overline{\cp^2}$. Let us write 
\begin{equation}\label{eqn:decomposition-wu}
N =( S^3 \times_{S^1} D^2 ) \cup (S^3 \times [0,1]) \cup_\varphi (S^3 \times_{S^1} D^2),
\end{equation}
using the diffeomorphism $\varphi: S^3 \times_{S^1} S^1 \to S^3$ given by $\varphi ([u,x]):= ux$. We now define a diffeomorphism $F:  N \to N$ as follows. On the first copy of $S^3 \times_{S^1} D^2$, $F$ is the identity. On the second copy of $S^3 \times_{S^1} D^2$, $F$ is given by complex conjugation $[u,x] \mapsto [\overline{u}, \overline{x}]$. Via the gluing diffeomorphisms, these maps induce the identity on $S^3 \times \{0\}$ and the complex conjugation $\kappa$ on $S^3 \times \{1\}$. The complex conjugation on $S^3 \subset \bC^2$ lies in $\SO(4)$, and hence there is a path $\gamma: [0,1] \to \SO(4)$ with $\gamma (0)=1$ and $\gamma (1)= \kappa$. We furthermore assume that $\gamma $ is constant near $0$ and constant near $1$. On the middle $S^3 \times [0,1]$, we define $F$ by the formula $F(t,u):= (t,\gamma(t)u)$. Altogether, this defines $F: N \to N$. 
We now define 
\[
 M:= W \cup_F W,
\]
which is a closed $5$-manifold which by construction is the union of two disc bundles of rank $3$, as claimed in Proposition \ref{prop:lastcase}. Since $W$ is simply connected and $N$ is connected, $M$ is simply connected, by Seifert--Van-Kampen. Therefore, $M$ is orientable, and we chose an orientation on $M$.
\end{construction}

\begin{proof}[Proof that $M$ is not nullbordant]
We will show that 
\begin{equation}\label{homology-wumanifold}
 H_k (M;\bZ) = 
\begin{cases}
\bZ & k =0,5,\\
\bZ/2 & k= 2,\\
0 & k \neq 0,2,5.
\end{cases}
\end{equation}
This will imply the lemma, by the following argument. It is enough to prove that $\langle w_2 (TM) w_3 (TM),[M] \rangle \neq 0$. To that end, recall that the \emph{semi-characteristic} $\kappa (X;\bF) \in \bZ/2$ of a closed $(2m+1)$-manifold $X$ with respect to a field $\bF$ is defined as
\[
\kappa(X;\bF):= \sum_{k=0}^{m} \dim_\bF (H_k (X;\bF)) \in \bZ/2. 
\]
Lusztig, Milnor and Peterson  showed in \cite{LuMilPet} that
\[
\kappa (X;\bQ)- \kappa (X;\bF_2) = \langle w_2 (TX) w_{2m-1} (TX),[X] \rangle \in \bZ/2. 
\]
Together with \eqref{homology-wumanifold}, this shows that $\langle w_2 (TM) w_3 (TM),[M] \rangle \neq 0$.

It remains to compute $H_2 (M;\bZ)$; the other groups are determined by the simple connectivity of $M$ and by Poincar\'e duality. The $S^1$-action on $S^2 \subset \bC \oplus \bR$ has the two fixed points $(0,\pm 1)$. Those two fixed points determine sections $s_\pm $ of the bundle $N=S^3 \times_{S^1} S^2 \to \cp^1$ and hence elements $b_\pm = \hur (s_\pm) \in H_2 (N;\bZ)$. It is easily verified (e.g. by the Mayer-Vietoris sequence for the decomposition \eqref{eqn:decomposition-wu}) that $(b_+,b_-)$ is a $\bZ$-basis of $H_2 (N) \cong \bZ^2$. 

The two $S^1$-fixed points $(0,\pm 1)$ can be connected by a path of fixed points in $D^3$; therefore, the two sections $s_\pm$ become homotopic in $S^3 \times_{S^1} D^3 = W$. Hence $j_* (b_-)= j_* (b_+) \in H^2 (W)$, where $j: N \to W$ denotes the inclusion map. Since $N \to W$ is $2$-connected, $j_*$ is surjective on $H_2$. Hence $j_* (b_\pm)$ must be a generator $u \in H_2 (W)\cong \bZ$. 

To compute the effect of the diffeomorphism $F$ on homology, let us agree that the signs are chosen in such a way that $s_-$ lies in the first part of the decomposition \eqref{eqn:decomposition-wu} and $s_+$ in the third one. With these choices
\[
 F_* (b_-)=b_-; \; F_*(b_+) = -b_+.
\]
The first is clear. For the second, consider the commutative diagram 
\[
 \xymatrix{
 S^3 \ar[d]^{\eta} \ar[r]^{\kappa} & S^3 \ar[d]^{\eta} \\
 \cp^1 \ar[r]^{q} & \cp^1
 }
\]
($q$ and $\kappa$ are the complex conjugation maps and $\eta$ the Hopf map). Using that $\deg (q)=-1$, we get that the complex conjugation induces $-1$ on $H_2 (S^3 \times_{S^1} 0) \cong H_2 (S^3 \times_{S^1} D^2)$. To compute $H_2 (M)$, consider the piece of the Mayer-Vietoris sequence for the decomposition $M = W \cup_F W$:
\[
 H_2 (N;\bZ) \stackrel{(j_*, j_* \circ F_*)}{\to} H_2 (W;\bZ) \oplus H_2 (W;\bZ) \to H_2 (M;\bZ) \to 0 
\]
With respect to the bases $(b_+,b_-)$ and $(u,u)$, the first map is represented by the matrix 
\[
\begin{pmatrix}
1 & 1 \\ -1 & 1 
\end{pmatrix}
\]
with determinant $2$. It follows that $ H_2 (M)\cong \bZ/2$, as claimed.
\end{proof}

\begin{rem}
It follows from Barden's classification theorem for simply connected $5$-manifolds \cite{barden65} that the manifold $M$ is diffeomorphic to  $\mathrm{SU}(3)/\SO(3)$, the Wu manifold. 
\end{rem}

It remains to be proven that $M= W \cup_F W$ is an admissible splitting. This is accomplished by Lemmas \ref{lemaa2:wumanifold} and \ref{lemma3:wumnaifild} below. To prove them, we have to do some geometry. 

\begin{lem}\label{lemaa2:wumanifold}
Let $\omega$ be the $U(2)$-invariant connection on the Hopf bundle $\eta:S^3 \to \cp^1$ which was introduced in \S \ref{subsec:projectivebundles}, and let $g_{\torp}^3$ be a torpedo metric on $D^3$ of radius $1$. Then the metric 
\[
 g_{\cp^1} \oplus_{\omega } g_{\torp}^3
\]
on $S^3 \times_{S^1} D^3= W$ has positive scalar curvature and is right stable. 
\end{lem}

\begin{proof}
By Theorem \ref{lemma:connectionmetricsrightstable}, it is enough to verify that $\scal ( g_{\cp^1} \oplus_{\omega } g_{\torp}^3) >0$. Using \eqref{eq:scalarcurvature-of-connectionmetric} and \eqref{eqn:scalar-curvature-projectivespace}, we get
\[
\scal ( g_{\cp^1} \oplus_{\omega } g_{\torp}^3) = 8 + \scal(g_{\torp}^3) - |A_{g_{\torp}^3}|^2.
\]
By definition of a torpedo metric, $\scal(g_{\torp}^3) \geq \scal(g_{S^2}) = 2$. The closed orbits of the $S^1$-action on $D^3$, in the torpedo metric $g_{\torp}^3$, have length at most $2 \pi$ and are therefore at most as long as the orbits of the translation action of $S^1$ on itself. It follows, using \eqref{eqn:computation-A-ternsor}, that 
\[
 |A_{g_{\torp}^3}|^2 \leq |A |^2= 2.
\]
where $|A|^2$ is the norm of the $A$-tensor for the translation action, which was computed in \eqref{eqn:A-torsor-hopf-bundles}.
\end{proof}

\begin{lem}\label{lemma3:wumnaifild}
The psc metrics $g_{\cp^1} \oplus_\omega g_{S^2}$ and $F^* (g_{\cp^1} \oplus_\omega g_{S^2})$ lie in the same path component of $\Riem^+ (N)$. 
\end{lem}

For the proof, we need another auxiliary result: 

\begin{lem}\label{lemma4:wumanifold}
Let $(N,g)$ be a closed Riemannian manifold, $\scal (g)>0$, and let $\gamma: [0,1] \to \Isom (N,g)$ be a smooth path which is constant near $0$ and $1$. Let $\Gamma: N \times [0,1] \to   N \times [0,1]$ be the diffeomorphism $(x,t) \mapsto (\gamma(t) x, t)$. Then the two metrics 
\[
 g \oplus dt^2, \Gamma^* (g \oplus dt^2) \in \Riem^+ (N \times [0,1])_{g,g}
\]
lie in the same path component. 
\end{lem}

\begin{proof}
Extend $\gamma$ to a smooth function $\gamma:\bR \to \Isom (N,g)$, constant on $(-\infty,0]$ and on $[1,\infty)$. For $r>0$, let $\gamma_r (t) := \gamma(\frac{1}{r} t)$ and define $\Gamma_r: N \times \bR \to N \times \bR$ accordingly. It will be enough to prove that for sufficiently large $r$, the metric 
\[
 (1-s)  g \oplus dt^2 + s\Gamma_r^* (g \oplus dt^2)
\]
on $N \times \bR$ has positive scalar curvature for all $s \in [0,1]$. Let $X_r$ be the vector field on $N \times \bR$ given by 
\[
 X_r (p,t):= \left. \frac{d}{du}\right|_{u=0} (\gamma_r (t + u) p ,t). 
\]
Note that 
\[
 X_r (p,t)= \frac{1}{r} X_1 (p, \frac{t}{r}). 
\]
Using that $\gamma (t)$ is an isometry, one calculates that the metric $\Gamma^* (g \oplus dt^2)$ is given (using the splitting $T_{(p,t)} (N \times \bR) = T_p N \oplus \bR$) by the matrix
\[
\begin{pmatrix}
g & X_r  \\ 
 X_r^{T} & 1+ \norm{X_r}^2 
\end{pmatrix}.
\]
For $r \to \infty$, this converges uniformly to $g \oplus dt^2$, together with all derivatives. From that, the claim follows.
\end{proof}

\begin{proof}[Proof of Lemma \ref{lemma3:wumnaifild}]
Let $g_\torp^2$ be a torpedo metric on $D^2$ of radius $1$. This does not have positive scalar curvature, but at least we know that $\scal (g_\torp^2) \geq 0$. Let $g_\dtorp^2$ be the metric on $D^2 \cup (S^1 \times [0,1]) \cup D^2 = S^2$ which is given by $g_\torp^2$ on the two discs and by $g_{S^1} \oplus dt^2$ on the cylinder. 
There is a smooth path $u \mapsto g(u)$ of Riemannian metrics on $S^2$ with $g(0)= g_{S^2}$, $g(1)= g_\dtorp^2$, such that each $g(u)$ is $S^1$-invariant, has nonnegative scalar curvature and such that the lengths of the $S^1$-orbits for each $g(u)$ are at most $2 \pi$. This can be seen from the formula for the scalar curvature of a warped product metric on $S^1 \times \bR$ which can be found in e.g. \cite[\S 2.2]{WalshMorse}. 

Consider the metric $g_{\cp^1} \oplus_\omega g(u)$ on $S^3 \times_{S^1} S^2 = N$. 
It follows that
\[
 \scal (g_{\cp^1} \oplus_\omega g(u) ) = 8 + \scal(g(u)) - |A_{g(u)}|^2 \geq 8 - |A|^2 =6
\]
for all $u$. Hence $g_{\cp^1} \oplus_\omega g_{S^2}$ and $g_{\cp^1} \oplus_\omega g_\dtorp^2$ lie in the same path component of $\Riem^+ (N)$, and therefore, it suffices for our aim to prove that 
\[
[ g_{\cp^1} \oplus_\omega g_\dtorp^2] = [F^*( g_{\cp^1} \oplus_\omega g_\dtorp^2)] \in \pi_0 (\Riem^+ (N)).
\]
By the construction of $F$ (and because the torpedo metric is $O(2)$-invariant), $F^*( g_{\cp^1} \oplus_\omega g_\dtorp^2)$ and $ g_{\cp^1} \oplus_\omega g_\dtorp^2$ agree on the two copies of $S^3 \times_{S^1} D^2$ inside $N$. On the middle piece $S^3 \times [0,1]$, we have 
\[
F^* ( g_{\cp^1} \oplus_\omega g_\dtorp^2) = F^* ( g_{\cp^1} \oplus_\omega (g_{S^1} \oplus dx^2)) = F^* (g_{S^3} \oplus dx^2)  \in \Riem^+ (S^3 \times [0,1])_{g_{S^3}, g_{S^3}}. 
\]
By Lemma \ref{lemma4:wumanifold}, it is isotopic within $\Riem^+ (S^3 \times [0,1])_{g_{S^3}, g_{S^3}}$ to the product metric. This finishes the proof. 
\end{proof}

\section{The computation in the non-spin case}\label{sec:non-spin}

The exact sequence \eqref{eqn:stolzkrecksequence} has an analogue for oriented cobordism, namely the exact sequence 
\[
\Omega^\SO_{d-4} (B (C_2 \ltimes PU(3) )) \stackrel{T}{\to} \Omega_d^\SO \to \pi_d (H \bZ),
\]
where $T$ takes the total space of $\cp^2$-bundles and the second map is given by the Thom class, see \cite{Fuehring}. Unfortunately, this is of no use for the proof of Theorem \ref{MainThm:nonspin}, since $\cp^2$ has too small dimension. Therefore, we have to take a more pedestrian route and need rather explicit generators for the oriented bordism ring \(\Omega_*^{\SO}\). The following summarizes the information which is relevant for us; we show in \S \ref{subsec:references} how to extract this from the classical literature.

\begin{thm}\label{thm:structure-oriented-bordism}
The oriented cobordism ring $\Omega_*^\SO$ is generated (multiplicatively) by elements of the following type. 
\begin{enumerate}
\item Classes $v_{4k} \in \Omega_{4k}^\SO$, one for each $k \geq 2$ and each $v_{4k}$ is a sum of classes which are represented by $\cp^n$-bundles with structure group $U(n+1)$ and $n \geq 3$, 
\item $v_4: = [\cp^2]$,
\item an infinite set $Z$ of classes in degrees $\geq 9$ which are represented by $\cp^n$-bundles with structure group $C_2$, $n \geq 3$,
\item a single element $z_5$ in degree $5$.
\end{enumerate}
The low-dimensional cobordism groups are as follows. 
\begin{enumerate}
\item $\Omega_k^\SO=0$ for $k=1,2,3,6,7$,
\item $\Omega_4^\SO = \bZ$, generated by $v_4=[\cp^2]$,
\item $\Omega_5^\SO \cong \bZ/2$, generated by $z_5$
\item $\Omega_8^\SO \cong \bZ^2$, generated by $v_8=[\cp^4]$ and $v_4^2=[\cp^2 ]^2$, and the subgroup generated by $[\cp^4]$ and $[\hp^2]$ has index $3$. 
\end{enumerate}
\end{thm}

\subsection{Proof of Theorem \ref{MainThm:nonspin}}

Let us first show how to use it to finish the proof of Theorem \ref{MainThm:nonspin}. Besides Theorem \ref{thm:structure-oriented-bordism}, we need one more computation from homotopy theory. 

\begin{prop}\label{prop:reductionstructuregroup-oriented}
The map 
\begin{equation}\label{eq:1}
\Omega_*^{\SO}(BU(1)^n)\rightarrow \Omega_*^{\SO}(BU(n)),
\end{equation}
induced by the inclusion of a maximal torus \(U(1)^n\rightarrow U(n)\), is surjective. Hence each $\cp^{n-1}$-bundle $E \to M$ over a closed oriented manifold with structure group $U(n)$ is oriented cobordant to a $\cp^{n-1}$-bundle with structure group $U(1)^n$. 
\end{prop}

\begin{proof}
The ring spectrum $\MSO$ is complex oriented; in fact the forgetful map $\MU \to \MSO$ defines a complex orientation, using Lemma II.4.6. of \cite{Adams}. By \cite[Proposition 4.3.3]{Kochman} or \cite[Lemma II.4.1]{Adams}, surjectivity of \eqref{eq:1} follows. 
\end{proof}

\begin{proof}[Proof of Theorem \ref{MainThm:nonspin}]
Let $A_d \subset \Omega_d^\SO$ be the subgroup of cobordism classes which contain a representative that has an admissible splitting. We first consider the case $d \neq 8$, $d \geq 5$, and have to prove that $A_d = \Omega_d^{\SO}$. 

Let $k \geq 2$ and let $v_{4k}$ be the generator described in Theorem \ref{thm:structure-oriented-bordism}. For an arbitrary $x \in  \Omega^\SO_*$, the element $xv_{4k}$ can be represented by a $\cp^{n-1}$-bundle with structure group $U(2) \times U(n-2)$, and $n\geq 4$, by Proposition 
\ref{prop:reductionstructuregroup-oriented}. Therefore, using Theorem \ref{thm:splitting-of-projective-bundles}, the ideal $(v_8,v_{12}, \ldots ) \subset \Omega_*^\SO$ is contained in $A_*$. An analogous argument shows that the ideal $(Z)$ ($Z$ as in Theorem \ref{thm:structure-oriented-bordism}) is contained in $A_*$. The only additive generators of $\Omega_*^{\SO}$ in dimensions $\geq 5$ which are missed by that argument are the elements of the form 
\[
z_5^a v_4^b, \; 5a+4b \geq 5. 
\]
The manifold $M$ in Proposition \ref{prop:lastcase} has an admissible splitting and since $[M] \neq 0$ and $\Omega_5^\SO \cong \bZ/2$, we have $[M]=z_5$ and so $z_5 \in A_5$. 

Propositions \ref{prop:lastcase}, together with Proposition \ref{cor:decomposables-have-admissibledecomposition}, implies that $z_5 x \in A_*$ when $x\in \Omega_*^\SO$ is of positive dimension. 

For $b\geq 3$, we write $v_4^b = [M \times \cp^2]$ with $\dim (M) \geq 6$ and use Proposition \ref{cor:decomposables-have-admissibledecomposition} to verify that $v_4^b \in A_*$.

These arguments cover all cases, except $d=8$. In that case, we certainly have $[\hp^2], [\cp^4] \in A_8$, and therefore the index of $A_8$ in $\Omega_8^\SO$ divides $3$, by Theorem \ref{thm:structure-oriented-bordism}. Hence if $[M] \in \Omega_8^\SO$, $M$ is cobordant modulo $A_8$ to either $W^8$, $\cp^2 \times \cp^2$, or $\cp^2 \times \cp^2$ with the reverse orientation, and so if $M$ is simply connected and not spinnable, $\Riem^+ (M)$ has the homotopy type of either $\Riem^+ (W^8)$ or $\Riem^+ (\cp^2 \times \cp^2)$. 
\end{proof}

\subsection{The structure of the oriented cobordism ring}\label{subsec:references}

This section is a guide through the literature on oriented cobordism. For a commutative ring $R$, a class $c \in H^* (BO;R)$ and $x \in \Omega_*^{\SO}$, we denote by $c(x) \in R$ the characteristic number $c(x)= \langle c(TM),[M]\rangle\in R$, where $[M] = x$. Recall furthermore the classes $s_n \in H^{4n}(BO;\bZ)$ from \cite[\S 16]{MilnorStasheff}, denoted $s_n (p)$ in loc.cit. 

\begin{thm}\label{thm:structuremso1}
Let $T_* \subset \Omega^{\SO}_*$ be the ideal of torsion elements. Assume that $v_{4k} \in \Omega_{4k}^{\SO}$, $k \geq 1$, are classes such that 
\begin{equation}\label{s-numbers}
s_k (v_{4k}) = \begin{cases}
 1 & 2k+1 \, \text{not a prime power}\\
 p & 2k+1 = p^s, \; s>0, \; p \; \text{prime},
\end{cases}
\end{equation}
and let $\overline{v_{4k}} \in \Omega_{4k}^{\SO}/T_{4k}$ be the residue class in the torsionfree quotient. Then 
\[
 \Omega_*^{\SO}/T_* = \bZ[\overline{v_{4k}} \vert k >0].
\]
Furthermore, $ 2T_*=0$.
\end{thm}
The computation of $\Omega_*^{\SO} / T_*$ is apparently due to Milnor (unpublished, but see \cite{Stong}, p. 207 for an account). Milnor \cite[Theorem 5]{MilnorCobordism} also proved that there is no odd primary torsion in $\Omega_*^{\SO}$, and Wall \cite[Theorem 2]{Wall60} showed that there is no element of order $4$. 

\begin{prop}\label{prop:torsionfreepart}
One can choose $v_4 = [\cp^2]$ and $v_8 = [\cp^4]$. For $k \geq 3$, one can choose $v_{4k}$ as a linear combination of total spaces of $\cp^n$-bundles with structure group $U(n+1)$, with $n \geq 3$. 
\end{prop}

\begin{proof}
If $2k+1$ is a prime, we can take $v_{4k}=[\cp^{2k}]$, because 
\[
 s_k (\cp^{2k})= 2k+1, 
\]
see \cite{HirzBergJung}, p. 42. This applies to $k=1,2$. If $k \geq 3$ and $2k+1$ is not a prime, we have to use further manifolds, the \emph{Milnor manifolds}. To define them, let $i \leq j$. The tautological line bundle $L_i \to \cp^i$ embeds canonically into the trivial vector bundle $\underline{\bC^{j+1}}$, and we let $L_{ij}^\bot$ be the orthogonal complement of $L_i$ in $\underline{\bC^{j+1}}$. The Milnor manifold is 
\[
 H_{ij} := \bP (L_{ij}^\bot). 
\]
By construction, $H_{ij}$ is a $\cp^{j-1}$-bundle over $\cp^i$ with structure group $U(j)$. Note that $H_{0j}=\cp^{j-1}$. The manifolds $H_{1j}$ are nullbordant. If $i \geq 2$ and $i+j-1 = 2k$, we have 
\[
 s_k (H_{ij}) = - \binom{i+j}{i}= - \binom{2k+1}{i}, 
\]
see \cite{HirzBergJung}, p. 43. 

For $k \geq 3$, $\cp^{2k}$ as well as $H_{ij}$ with $i+j=2k+1$, $i \geq 2$ and $j \geq 4$, are bundles of projective space as asserted. Therefore, by taking a suitable linear combination of these manifolds, we get a class $v_{4k} \in \Omega^{\SO}_{4k}$ with 
\[
 s_k (v_{4k})= \gcd \{ 2k+1, \binom{2k+1}{2k+1-j}, 4 \leq j  \leq 2k-1 \} . 
\]
If $k \geq 3$, this number is the same as 
\[
 d_{2k+1} := \gcd \{ \binom{2k+1}{n} \vert 1 \leq n \leq 2k \}.
\]
It is a rather unknown, but elementary fact \cite{Ram} that 
\[
\gcd \{ \binom{n}{j} \vert 1 \leq j \leq n-1\} = \begin{cases}
1 & n \, \text{not a prime power}\\
 p & n = p^s, \; s>0, p \; \text{prime}.       
 \end{cases}
\]
From that, the claim follows. 
\end{proof}

Wall also determined the structure of $T_*$ to which we turn now. We denote by $[M]^O \in \Omega_n^{O}$ the unoriented cobordism class of a closed $n$-manifold $M$. 
Following \cite{Wall60}, let $\mathfrak{W}_* \subset \Omega_*^O$ be the set of all (unoriented) cobordism classes which have representatives $M$ such that the Stiefel-Whitney class $w_1 (TM) \in H^1 (M;\bZ/2)$ are the $\pmod 2$ reduction of an integral class $w \in H^1 (M;\bZ)=[M;S^1]$. This is a subalgebra of $\Omega_*^O$ \cite[Lemma 3]{Wall60}, and the forgetful map $\rho:\Omega_*^{\SO} \to \Omega_*^O$ factors through $\mathfrak{W}_*$. 

Let $u: M^n \to S^1$ be a smooth map representing $w$ and pick a regular value $a$ of $u$. The manifold $V := u^{-1}(a)$ is orientable, and $2[V]=0 \in \Omega_{n-1}^{\SO}$ \cite[Lemma 1]{Wall60}. Theorem 3 of \cite{Wall60} says that the assignment $M \mapsto V$ gives a well-defined map
\[
\partial : \mathfrak{W}_* \to \Omega^{\SO}_{*-1},
\]
which is a derivation and fits into a long exact sequence
\[
\ldots \to \Omega_n^{\SO} \stackrel{2}{\to} \Omega_n^{\SO} \stackrel{\rho}{\to} \mathfrak{W}_n \stackrel{\partial}{\to} \Omega_{n-1}^{\SO} \to \ldots 
\]
This and the fact that $2 T_*=0$ implies $T_* = \im (\partial)$. For the description of the structure of $\mathfrak{W}_*$, we need the following manifolds. Let $P(m,n)$ be the Dold manifold $S^m \times_{C_2} \cp^n \to \rp^m$, where $C_2$ acts on $\cp^n$ by complex conjugation. 
Furthermore, let $A: P(m,n) \to P(m,n)$ be the map induced by $(x,z) \mapsto (Rx,z)$, where $R$ is a reflection along a hyperplane. We define $Q(m,n)$ as the mapping torus of $A$; there is an obvious fibre bundle $u:Q(m,n) \to S^1$ with fibre $P(m,n)$. 
If $m$ is odd and $n$ is even, $P(m,n)$ is orientable, and $A$ reverses orientation, and $u$ defines an integral lift of $w_1 (TQ(m,n))$ in this situation. Hence $[P(m,n)]^O$ and $[Q(m,n)]^O$ are elements of $\mathfrak{W}_*$, if $m$ is odd and $n$ is even. 

\begin{thm}\label{thm:structuremso2}
If $k$ is not a power of $2$, write $k=2^{r-1} (2s+1)$, $s \neq 0$, and define 
\[
y_{2k-1}:= [P(2^r-1,2^r s)]^O \in \mathfrak{W}_{2k-1}
\]
and 
\[
y_{2k}:= [Q(2^r-1,2^r s)]^O \in \mathfrak{W}_{2k}.
\]
Let furthermore $y_{2^r}:=[\cp^{2^{r-1}}]^O \in \mathfrak{W}_{2^r}$. Then $\mathfrak{W}_*$ is a polynomial algebra over $\bF_2$, with generators the elements $y_{2k}, y_{2k-1}$ ($k$ not a power of $2$) and $y_{2^r}$, $r \geq 2$. 
\end{thm}

\begin{proof}[Short guide through the proof]
This is shown in \cite[\S 5-6]{Wall60}, but not explicitly stated there, hence we recall the argument, following the streamlined version \cite{AtiyahCobordism} of Wall's proof by Atiyah. 
Recall that Thom \cite[Th\'eor\`eme IV.12]{Thom} proved that $\Omega_*^O$ is a polynomial algebra over $\bF_2$, with one generator $x_i$ for each $i$ which is not of the form $2^r-1$, see also \cite[p. 96]{Stong}. Combining Thom's theorem with \cite[Satz 3]{Dold} and the proof of \cite[Lemma 6]{Wall60}, the elements in the theorem are algebraically independent, and hence generate a polynomial subalgebra $\mathfrak{V}_* \subset \mathfrak{W}_*$. The Hilbert--Poincar\'e series $HS_{\mathfrak{V}_*} (t)=\sum_{n=0}^\infty \dim_{\bF_2} (\mathfrak{V}_n) t^n \in \bZ[[t]]$ is easily calculated. 

Using formula (9) of \cite{AtiyahCobordism} and the structure of $\Omega_*^O$, one computes the Hilbert--Poincar\'e series $HS_{\mathfrak{W}_*} (t)$ and verifies that both series agree, whence $\mathfrak{V}_*=\mathfrak{W}_*$.
\end{proof}

The following finishes the description of the multiplicative generators of $\Omega_*^\SO$.

\begin{cor}\label{cor:torsioninMSO}
Let $x \in T_*$ be a torsion element in $\Omega_*^{\SO}$. Then $x$ is a linear combination of manifolds, which are either 
\begin{enumerate}
\item $\cp^n$-bundles with structure group $C_2$ and $n \geq 3$, 
\item or products of $P(1,2)$ with some other oriented manifold. 
\end{enumerate}
\end{cor}

\begin{proof}
Since $T_* = \im (\partial)$, it is enough to check that all monomials in the generators of $\mathfrak{W}_*$, are mapped to cobordism classes of the claimed form under $\partial$. 

Let $k, r$ and $s$ as in Theorem \ref{thm:structuremso2}. The elements $y_i$ introduced there have the preimages 
\[
z_{2k-1}:= [P(2^r-1,2^r s)] \in \Omega^{\SO}_{2k-1}
\]
and 
\[
z_{2^r}:= [\cp^{2^{r-1}}] \in \Omega_{2^r}^{\SO}
\]
under $\rho: \Omega_*^{\SO} \to \mathfrak{W}_*$. Therefore $\partial (y_{2k-1})=0$ and $\partial (y_{2^r})=0$. 

Since $\partial: \mathfrak{W}_* \to \Omega_*^{\SO}$ is a derivation
\[
\partial ( (\prod_i  y_{2k_i-1}^{a_i}) (\prod_i  y_{2^{r_i}}^{b_i}) (\prod_i y_{2k_i}^{c_i}) )= \\
 (\prod_i  y_{2k_i-1}^{a_i}) (\prod_i  y_{2^{r_i}}^{b_i})  \partial  (\prod_i y_{2k_i}^{c_i}) ).
\]
The Dold manifold $P(2^r-1,2^r s)$ is a $\cp^{2^r s}$-bundle, and $2^r s \geq 2$, with equality only if $r=s=1$, i.e. $k = 3$. The element $z_5$ is represented by $P(1,2)$. 
For all other values of $k$, $z_{2k-1}$ is represented by a bundle of complex projective spaces of dimension at least $3$. 

Furthermore, $ \partial  (\prod_i y_{2k_i}^{c_i}) )$ is represented by a manifold $V$ of the following type. Take a product $Q(m_1,n_1) \times \ldots \times Q(m_p,n_p) \to (S^1)^i$, which is a fibre bundle with fibre a product of Dold manifolds, restrict this fibre bundle to the kernel $\Delta \subset (S^1)^i$ of the multiplication map $\mu:(S^1)^p \to S^1$, and take the total space. 

Now there is a fibre bundle $Q(m,n) \to Q(m,0)$ with fibre $\cp^n$ and structure group $C_2$ (complex conjugation). This is compatible with the map to $S^1$. Therefore, $V$ is the total space of a bundle of projective spaces with structure group $C_2$. The complex dimension of the fibre can be chosen to be at least $3$, except in the case of $y_6^c$. If $c$ is even, then $\partial (y_6^c)=0$, and furthermore
\[
\partial (y_6^{2m+1})= z_5 z_6^{2m}.
\]
The claim follows.
\end{proof}

From the description of the generators, it is also apparent that $\Omega_5^\SO\cong \bZ/2$, that $\Omega_6^\SO=\Omega_7^\SO=0$ and that $\Omega_8^\SO$ has the $\bZ$-basis $[\cp^4]$ and $[\cp^2 \times \cp^2]$. Moreover
\[
[\hp^2 ]= 3 [\cp^2 \times \cp^2] - 2 [\cp^4]: 
\]
both sides of the equation have the same signature; furthermore $s_2 ([\cp^2 \times \cp^2])=0$, $s_2 ([\cp^4])= 5$ and $s_2 ([\hp^2])=-10$, which is easily derived from the computations in \cite{HirzBergJung}, p.5 f. Therefore, $[\cp^4]$ and $[\hp^2]$ span a subgroup of index $3$ in $\Omega_8^{\SO}$.

\bibliographystyle{plain}
\bibliography{homotopyinvariance}

\end{document}